\documentclass[11pt]{amsart}

\usepackage{amssymb,amsfonts,amsmath,amsthm,xcolor}

\newtheorem{theorem}{Theorem}[section]
\newtheorem{lemma}[theorem]{Lemma}
\newtheorem{proposition}[theorem]{Proposition}
\newtheorem{proposition-and-definition}[theorem]{Proposition and Definition}
\newtheorem{proposition-and-notation}[theorem]{Proposition and Notation}
\newtheorem{corollary}[theorem]{Corollary}

\theoremstyle{definition}
\newtheorem{definition}[theorem]{Definition}
\newtheorem{notation}[theorem]{Notation}
\newtheorem{remark}[theorem]{Remark}
\newtheorem{notation-and-remark}[theorem]{Notation and Remark}
\newtheorem{definition-and-remark}[theorem]{Definition and Remark}
\newtheorem{remark-and-notation}[theorem]{Remark and Notation}
\newtheorem{remark-and-definition}[theorem]{Remark and Definition}
\newtheorem{example}[theorem]{Example}

\newtheorem{ad-hoc-item}[theorem]{  }

\newcommand{\cA}{\mathcal{A}}
\newcommand{\cB}{\mathcal{B}}

\newcommand{\bC}{ \mathbb{C} }

\newcommand{\cH}{ \mathcal{H} }
\newcommand{\cI}{ \mathcal{I} }
\newcommand{\cJ}{ \mathcal{J} }
\newcommand{\cM}{ \mathcal{M} }

\newcommand{\bN}{ \mathbb{N} }
\newcommand{\cP}{ \mathcal{P} }
\newcommand{\bR}{ \mathbb{R} }

\newcommand{\limn}{ \lim_{n \to \infty} }

\newcommand{\Ker}{ \mathrm{ker} }

\newcommand{\cyc}{\eta}
\newcommand{\symp}{\mathsf{p}}
\newcommand{\llambda}{\mathsf{p}}
\newcommand{\vvarphi}{\chi}
\newcommand{\muw}{ \mu_{\uw} }

\newcommand{\linspan}{\mathrm{span}}

\newcommand{\tr}{ \mathrm{tr} }

\newcommand{\WOT}{ \mathrm{WOT} }
\newcommand{\uee}{\ee }
\newcommand{\ui}{ \underline{i} }
\newcommand{\uj}{ \underline{j} }
\newcommand{\ur}{ \underline{r} }
\newcommand{\uw}{ \underline{w} }
\newcommand{\bt}{ \mathbf{t} }
\newcommand{\bu}{ \mathbf{u} }

\newcommand{\ans}{ (a_n)_{n=1}^{\infty} }
\newcommand{\oneA}{ 1_{{ }_{\cA}} }
\newcommand{\oneM}{ 1_{{ }_{\cM}} }

\newcommand{\diamondU}{\overset{\diamond}{U}}

\newcommand{\ee}{\varepsilon}
\newcommand{\wx}{\omega}
\newcommand{\la}{\langle}
\newcommand{\ra}{\rangle}

\allowdisplaybreaks

\numberwithin{equation}{section}

\title[CLT for star-generators of $S_{\infty}$,
and CCR]{A central limit theorem for star-generators of 
\boldmath{$S_{\infty}$}, which relates to traceless CCR-GUE matrices} 

\author[J. Campbell]{Jacob Campbell}
\address{Jacob Campbell: Department of Pure Mathematics, 
University of Waterloo, Ontario, Canada.}
\email{j48campb@uwaterloo.ca}
\author[C. K\"ostler]{Claus K\"ostler}
\address{Claus K\"ostler: School of Mathematical Sciences,
University College Cork, Ireland.}
\email{claus@ucc.ie}
\author[A. Nica]{Alexandru Nica}
\thanks{AN: research supported by a Discovery Grant from 
NSERC, Canada.}
\address{Alexandru Nica: Department of Pure Mathematics, 
University of Waterloo, Ontario, Canada.}
\email{anica@uwaterloo.ca}

\begin{document}

\begin{abstract}
$\ $ $\ $ We prove a limit theorem concerning the sequence of star-generators 

\noindent
of $S_{\infty}$, where the expectation 
functional is provided by a character of 
$S_{\infty}$ with weights $(w_1, \ldots , w_d, 0,0, \ldots )$
in the Thoma classification. The limit law turns out to be
the law of a ``traceless CCR-GUE'' matrix, an analogue
of the traceless GUE where the off-diagonal entries 
$g_{i,j}$ satisfy the commutation relation
$g_{i,j}g_{j,i} = g_{j,i}g_{i,j} + (w_j - w_i)$.  The 
special case $w_1 = \cdots = w_d = 1/d$ yields the law of
a bona fide traceless GUE matrix, and we retrieve a result 
of K\"ostler and Nica from 2021,
which in turn extended a result of Biane from 1995.
\end{abstract}

\maketitle
\section{Introduction}

\noindent
The sequence of star-generators 
$(1,2), (1,3), \ldots , (1,n), \ldots$ of the infinite symmetric 
group $S_{\infty}$ has received a good amount of attention in recent 
years.  This was due to nice combinatorial facts found about   
factorizations of arbitrary permutations into products of star-generators,
and also because of the direct connection that star-generators have with
the important family of Jucys-Murphy elements of $\bC [S_{\infty} ]$.
See for instance the presentation of results shown in the introduction 
to \cite{Fe2012}.

An interesting point of view, going back to a paper of 
Biane \cite{Bi1995}, is to treat the star-generators as a sequence
of selfadjoint random variables in the $*$-probability space 
$( \bC [ S_{\infty} ] , \varphi )$, where $\varphi$ is the
canonical trace on the group algebra $\bC [ S_{\infty} ]$.  In
particular, one of the results in \cite{Bi1995} is a limit theorem 
with CLT flavour which holds in this framework, and finds the 
semicircle law of Wigner to be the limit law.

The canonical trace $\varphi$ on $\bC [S_{\infty}]$ is the 
$d \to \infty$ limit of a significant sequence of trace-states 
$\varphi_d : \bC[ S_{\infty} ] \to \bC$, $d \in \bN$, given
by the so-called ``block characters'' of $S_{\infty}$ (see 
e.g.~the presentation in \cite{GnGoKe2013}).  
In the recent paper \cite{KoNi2021} it was observed that the 
CLT theorem for the star-generators of $S_{\infty}$ still holds 
in the $*$-probability space $( \bC [S_{\infty} ], \varphi_d )$ 
for a finite value of $d \in \bN$.  Moreover, the limit law for 
the CLT can in this case be identified as the law of a known 
$d \times d$ random Hermitian matrix, the ``traceless GUE matrix''.

Moving one step further, we note that the block character which
defines $\varphi_d$ is a special case of extremal character
of $S_{\infty}$.  In the classification of such extremal characters 
that was given by Thoma \cite{Tho1964}, this block character is 
encoded by the sequence 
\begin{equation}   \label{eqn:1a}
(1/d, \ldots , 1/d, 0, \ldots , 0, \ldots ), 
\ \ \mbox{ with $d$ occurrences of $1/d$ at the beginning.}
\end{equation}
(The Thoma classification actually encodes a character by using 
two sequences of numbers in $[0, \infty )$; but in this paper 
we are only dealing with situations where the second Thoma sequence 
has all its terms equal to $0$.)

The present paper continues the study of the CLT theorem 
for star-generators, in the setting where instead of the 
sequence (\ref{eqn:1a}), we consider a sequence 
\begin{equation}   \label{eqn:1b}
\left\{ 
\begin{array}{c}
(w_1, \ldots , w_d, 0, \ldots , 0, \ldots ), 
\mbox{ with $d \geq 2$ and with} \\
\mbox{$w_1 \geq w_2 \geq \cdots \geq w_d > 0$ 
such that $\sum_{i=1}^d w_d = 1$.}
\end{array}  \right.
\end{equation}
As explained below, the CLT result still holds in 
connection to a tuple of weights as in (\ref{eqn:1b}), 
and continues to belong to the general category of 
``exchangeable CLT'' theorems.  This is stated as 
Theorem \ref{thm:25} of the paper, which we prove 
by reduction to a basic exchangeable CLT result of 
Bo\.zejko and Speicher \cite{BoSp1996}. 
Moreover, the limit law $\muw$ arising in 
Theorem \ref{thm:25} still has a neat realization 
as the law of an appropriate version of $d \times d$ 
traceless GUE matrix.  But the extension from the 
framework of (\ref{eqn:1a}) to the one of (\ref{eqn:1b}) 
is non-trivial in the following two respects.

\vspace{6pt}

\noindent
(a) When looking at the law of large numbers that typically
precedes a CLT, one finds the case 
$w_1 = w_2 = \cdots = w_d = 1/d$ to be the only one where the
centering of the star-generators is done in the usual way, by
subtracting a scalar multiple of the unit of $\bC [S_{\infty}]$. 
For any other tuple $(w_1, \ldots , w_d)$, the centering has
to be done by subtracting a (non-scalar) operator $A_0$ in 
the GNS representation of the character; this naturally moves
the framework of the CLT to the $*$-probability $( \cM , \tr )$,
where $\cM$ is the von Neumann algebra generated 
by the said GNS representation, and $\tr$ is the natural trace
that $\cM$ comes equipped with.  We explain how the centering goes
in Section 2.2, and then discuss this in more detail in Section 3 
of the paper.

\vspace{6pt}

\noindent 
(b) The CLT limit law $\mu_{\uw}$ corresponding to a general 
tuple $\uw = (w_1, \ldots , w_d)$ of weights as in 
(\ref{eqn:1b}) still is realized as the law of a special
$d \times d$ matrix $M$; but the entries of $M$ are now 
living in the non-commutative world of canonical commutation
relations (CCR).  More precisely: for $1 \leq i,j \leq d$ with
$w_i \neq w_j$, the $(i,j)$-entry $a_{i,j}$ of $M$ 
and its adjoint $a_{j,i} = a_{i,j}^{*}$ are set to 
satisfy the relation
$a_{i,j} a_{j,i} = a_{j,i} a_{i,j} + (w_j - w_i)$.  
This forces $a_{i,j}$ to no longer be a
random variable in the usual sense.  But nevertheless, 
the description of $a_{i,j}$ clearly fits within a notion of
``CCR-complex-Gaussian random variable'', and the resulting 
matrix $M$ is a natural CCR-analogue for a traceless GUE 
matrix.  These points are explained in Section 2.3 below, 
where we also state the main result of the paper, 
Theorem \ref{thm:29}.

$\ $

\noindent
{\bf Organisation of the paper.}
Besides the present Introduction, the paper
has 7 sections.  

\vspace{4pt}

\noindent
-- Section 2 gives a general presentation of the framework 
and results of the paper.  The CLT result is stated as
Theorem \ref{thm:25} in Section 2.2, then in Section 2.3 
we introduce the notion of traceless CCR-GUE matrix and we 
give the statement of Theorem \ref{thm:29}.  Section 2.4
gives a glimpse of the underlying combinatorics which 
connects Theorems \ref{thm:25} and \ref{thm:29}: the proofs
of both these theorems require fiddling with 
pair-partitions $\pi$ of sets $\{ 1, \ldots , k \}$, only 
that this fiddling is done in two different ways, which 
associate two different permutations, 
``$\tau_{\pi}$ vs. $\sigma_{\pi}$'',
to the same pair-partition $\pi$.  
Proposition \ref{prop:211} of Section 2.4 explains how
the cycle structure of $\tau_{\pi}$ can be read from 
$\sigma_{\pi}$ -- this is a key-point towards 
the proof of Theorem \ref{thm:29}.

\vspace{4pt}

\noindent
-- In Section 3 we set the framework used throughout the 
paper, and we review the law of large numbers which introduces
the operator $A_0$ used for centering in our CLT result.  

\vspace{4pt}

\noindent
-- In Section 4 we review the setting for
an exchangeable CLT, and we prove Theorem \ref{thm:25}.

\vspace{4pt}

\noindent
-- Sections 5-7 are devoted to studying the moments of the 
limit law $\muw$ which arises in Theorem \ref{thm:25}.  
More precisely: in Section 5 we clarify 
(cf.~Proposition \ref{prop:56}) how the said moments 
are expressed in terms of the permutations $\tau_{\pi}$.
In Section 6 we present in more detail the connection 
$\tau_{\pi} \, \leftrightarrow \, \sigma_{\pi}$
advertised in the presentation of results from Section 2.4. 
This is then used in Section 7 
in order to obtain a formula where the moments of $\muw$
are written directly in terms of permutations 
$\sigma_{\pi}$ (cf.~Proposition \ref{prop:76}).

\vspace{4pt}

\noindent
-- Finally, in Section 8 we discuss in more detail the
notion of traceless CCR-GUE matrix.  We obtain 
(cf.~Proposition \ref{prop:84}) a Wick-style formula for 
the joint moments of the entries of such a matrix, and 
this Wick-style formula is then used to prove 
the main result of the paper, Theorem \ref{thm:29}.

$\ $

\section{Presentation of results}

\subsection{Description of framework, the law of large numbers.}

\begin{notation}   \label{def:21}
{\em (Some general notation.)}
As is customary, we denote by $S_{\infty}$ the group of all finite
permutations $\sigma$ of $\bN = \{ 1,2, \ldots , n , \ldots \}$ 
(thus $\sigma : \bN \to \bN$ is bijective, and there exists $n_o \in \bN$
such that $\sigma (n) = n$ for $n>n_o$).
We will write permutations $\sigma \in S_{\infty}$ in cycle 
notation, where $\sigma$ is expressed as a product of disjoint cycles, 
and every $n \in \bN$ not specifically included in a cycle is assumed 
to be a fixed point of $\sigma$.
Important instances of such writing are provided by the sequence of 
{\em star-transpositions} $(\gamma_n )_{n=1}^{\infty}$
defined as
\begin{equation}   \label{eqn:21a}
\gamma_1 = (1,2), \, \gamma_2 = (1,3), \ldots , 
\gamma_n = (1,n+1), \ldots ,
\end{equation}
and by the {\em forward cycles} $(\cyc_n)_{n=1}^{\infty}$, which are 
\begin{equation}   \label{eqn:21b}
\cyc_1 = (1) \, ( = \mbox{unit of $S_{\infty}$)}, 
\, \cyc_2 = (1,2), \, \cyc_3 = (1,2,3), \ldots , 
\cyc_n = (1,2, \ldots , n), \ldots
\end{equation}
\end{notation}

\vspace{6pt}

\begin{notation}   \label{def:22}
{\em (The character $\chi$.)}
Throughout the whole paper we fix a $d \geq 2$ and a tuple 
of weights,
\begin{equation}    \label{eqn:22a}
\uw = (w_1, \ldots , w_d), 
\end{equation} 
with $w_1 \geq w_2 \geq \cdots \geq w_d >0 \mbox{ and }  
w_1 + \cdots + w_d =1$.
We consider the power sums
\begin{equation}    \label{eqn:22b}
\symp_n := w_1^n + w_2^n + \cdots + w_d^n, \ \ n \in \bN,
\end{equation}
thus getting a sequence of numbers
$1 = \symp_1 > \symp_2 > \cdots > \symp_n > \cdots > 0$.
With this in hand, we then define 
$\vvarphi : S_{\infty} \to \bR$ by putting
\begin{equation}    \label{eqn:22c}
\vvarphi ( \sigma ) := \prod_{ 
     \substack{V \, \mathrm{orbit \ of} \ \sigma, \\
               |V| \geq 2} } \ \symp_{|V|}.
\end{equation}
(Concrete example: $\sigma = (1,3,2)(5,6) \in S_{\infty}$
has $\vvarphi ( \sigma ) = \symp_2 \cdot \symp_3$.  In the
case when $\sigma = \cyc_1$ = unit of $S_{\infty}$, the empty
product on the right hand side of (\ref{eqn:22c}) is taken
to be equal to $1$.)

It is immediate that $\vvarphi$ is a class-function, that is,
the value $\vvarphi ( \sigma )$ only depends on the conjugacy 
class of $\sigma$ in $S_{\infty}$.  It is, moreover, not hard
to prove that $\vvarphi$ is a function of positive type, hence 
it is what one calls a {\em character} of the group $S_{\infty}$.
Some more advanced considerations (see 
e.g.~\cite[Section 4.2]{BoOl2017}) show that $\vvarphi$ is 
actually an {\em extremal character} of $S_{\infty}$, appearing
in the Thoma classification of such characters in the way indicated 
in (\ref{eqn:1b}) above.
\end{notation}

\vspace{6pt}

\begin{notation}    \label{def:23} 
{\em (GNS.)}  We will use the notation
$U : S_{\infty} \to B( \cH )$ for the GNS representation of
$\vvarphi$.  Thus $\cH$ is a Hilbert space over $\bC$,
endowed with a map
\begin{equation}    \label{eqn:23a}
S_{\infty} \to \cH : \sigma \mapsto \widehat{\sigma},
\end{equation}
such that
\begin{tabular}[t]{ll}
(1) & the linear span of the image of the map 
      (\ref{eqn:23a}) is dense in $\cH$;  \\
(2) & for every $\sigma,\tau \in S_{\infty}$ we have 
$\langle \widehat{\sigma}, \widehat{\tau} \rangle 
= \vvarphi ( \sigma \tau^{-1} )$.
\end{tabular}

\noindent
(Since $\vvarphi$ is a class-function, the 
$\sigma \tau^{-1}$ in (2) can be replaced by any of 
$\tau  \sigma^{-1}$, $\sigma^{-1} \tau$, or $\tau^{-1} \sigma$.)
Then, for every $\sigma \in S_{\infty}$, the operator 
$U( \sigma ) \in B( \cH )$ is defined via the requirement that
\begin{equation}   \label{eqn:23b}
\bigl[ U( \sigma ) \bigr] \, ( \widehat{\tau} ) 
= \widehat{ \sigma \tau }, \ \ \forall
\, \tau \in S_{\infty}.
\end{equation}
It is easily verified that the definition of $U ( \sigma )$ makes
sense and that the map $\sigma \mapsto U( \sigma )$ is a 
representation of $S_{\infty}$ by unitary operators on $\cH$.
Note that for a star-transposition $\gamma_n$, the operator 
$U( \gamma_n )$ is a symmetry, so in particular it is self-adjoint.
\end{notation}

\vspace{6pt}

\begin{remark}    \label{rem:24}
{\em (Law of large numbers.)}  In our setting, this amounts
to the fact that the averages of $U( \gamma_n )$'s are 
convergent in the strong operator topology.  We denote
\begin{equation}   \label{eqn:24a}
\mathrm{SOT}-\lim_{n \to \infty}
\frac{1}{n} \Bigl( U( \gamma_1 ) + \cdots + U( \gamma_n ) \Bigr)
=: A_0 \in B( \cH ).
\end{equation}
The operator $A_0$ was studied in detail in \cite{GoKo2010}, 
and played a crucial role in the considerations of non-commutative 
dynamics made in that paper.  For the sake of keeping the 
presentation self-contained, we include in Section 3 below the
relatively easy proofs of the facts about $A_0$ needed here,
in particular the convergence in (\ref{eqn:24a}) and the relevant
fact that 
\begin{equation}   \label{eqn:24b}
\mathrm{Spectrum}(A_0) 
= \{ t \in (0,1) \mid \, \exists \, 1 \leq i \leq d 
\mbox{ such that } w_i = t \}.
\end{equation}
From (\ref{eqn:24b}) it follows that $A_0$ is an invertible 
selfadjoint operator -- but note that the case 
$w_1 = \cdots = w_d =1/d$ is the only one when $A_0$ comes out
as a scalar multiple of the identity operator on $\cH$.
\end{remark}

$\ $

\subsection{The theorem of CLT type, and the limit law \boldmath{$\muw$}.}

$\ $

\noindent
The next step to be taken, once the limit $A_0$ from
(\ref{eqn:24a}) was identified, is to subtract $A_0$ out of 
the $U( \gamma_n )$'s and seek a limit in law for the rescaled
averages
\begin{equation}   \label{eqn:25a}
\frac{1}{\sqrt{n}} \, \sum_{i=1}^n (U(\gamma_i) - A_0) \text{.} 
\end{equation}
The law of the operators (\ref{eqn:25a}) is to be considered 
with respect to the vector-state defined by the vector 
$\widehat{\cyc_1} \in \cH$, where recall that (coming from 
(\ref{eqn:21b})) we use the notation ``$\cyc_1$'' for the unit
of $S_{\infty}$.  Equivalently, the law of these operators is to 
be considered in the $W^{*}$-probability space $( \cM , \tr )$,
where $\tr$ is the natural trace-state of the von Neumann algebra 
$\cM \subseteq B( \cH )$ generated by the $U( \gamma_n )$'s --
cf.~discussion in Section 3.1 below.

It turns out that the following holds.

\vspace{6pt}

\begin{theorem}    \label{thm:25}
For every $n \in \bN$, let $\mu_n$ denote the law of the 
rescaled average (\ref{eqn:25a}), in the sense described above. 
Then, when $n \to \infty$, the probability measures $\mu_n$ 
have a $\mathrm{weak}^{*}$-limit $\muw$, which depends on
the tuple of weights $\uw = (w_1, \ldots , w_d)$ fixed in 
Notation \ref{def:22}.
\end{theorem}

\vspace{6pt}

Theorem \ref{thm:25} falls under the general umbrella of 
``exchangeable CLT'' results.  Indeed, it is easily seen that 
upon putting
\begin{equation}   \label{eqn:25b}
\diamondU_n :=
U( \gamma_n ) - A_0 \in \cM, \ \ n \in \bN,
\end{equation}
we get a sequence $(\,  \diamondU_n )_{n=1}^{\infty}$ of elements of 
$\cM$ which are centred and exchangeable with respect to the trace 
$\tr$.  In Section 4 of the paper we provide an elementary 
combinatorial proof that the operators $\diamondU_n$ also satisfy
the ``vanishing singleton property'' used as hypothesis in a basic 
exchangeable CLT result of Bo\.zejko and Speicher, Theorem 0 in 
\cite{BoSp1996}.  This makes Theorem \ref{thm:25} come out as 
a corollary of the said theorem from \cite{BoSp1996}.  An
alternative, more high-powered approach is to obtain
Theorem \ref{thm:25} by reducing it to
\cite[Theorem 9.4]{Ko2010}, a result which applies to sequences of 
non-commutative random variables displaying weaker probabilistic 
symmetries than what we have here.  

We emphasize the fact that in order for the CLT mechanism to 
kick in, it is important that the centering of the 
$U( \gamma_n )$'s goes with
$\diamondU_n := U( \gamma_n ) - A_0$.  This differs from 
the usual centering procedure, which is done by subtracting 
a scalar multiple of the unit, and is commonly denoted by 
putting a circle on the top of the element which is being 
centered.  In the case at hand we would write
$\overset{\circ}{U}_n := U( \gamma_n ) - \lambda \oneM$, with
$\lambda = \tr \bigl( \, U( \gamma_n ) \, \bigr) = \symp_2$.
But unless we are in the special case with 
$w_1 = \cdots = w_d = 1/d$, using the operators 
$\overset{\circ}{U}_n$ does not lead to a theorem of CLT type.

$\ $

\subsection{Realization of \boldmath{$\muw$} as the law 
of a traceless CCR-GUE matrix.}

$\ $

\noindent
The GUE model is one of the most referenced models of random 
Hermitian matrix -- see e.g.~\cite[Chapter 3]{AnGuZe2009}.  For
the purposes of this paper, it is convenient to consider a 
GUE matrix $G = [g_{i,j}]_{i,j=1}^d$ with entries rescaled such
that the expected normalized trace of $G^2$ is 
$E ( \tr_d (G^2) ) = 1$.  The so-called ``traceless GUE'' is
less referenced in the literature, but is for instance 
mentioned in Biane's paper \cite{Bi1995} cited above.  It is
the random matrix $M$ obtained from the above $G$ by projecting 
the random vector $(g_{1,1}, \ldots , g_{d,d}) \in \bR^d$ onto 
the hyperplane of equation $t_1 + \cdots + t_d = 0$. Thus
\begin{equation}   \label{eqn:23x}
M := G - \frac{g_{1,1} + \cdots + g_{d,d}}{d} \, I_d,
\end{equation}
where $I_d$ is the identity $d \times d$ matrix.  It follows, 
in particular, that the diagonal entries of $M$ are linearly 
dependent.  They form a Gaussian family of centred random 
variables with covariance matrix 
$C = [c_{i,j}]_{i,j=1}^d$, where: 

\begin{equation}   \label{eqn:23xx}
c_{i,i} = (d-1)/d^2 \mbox{ for $1 \leq i \leq d$ and }
c_{i,j} = c_{j,i} = - 1/d^2 \mbox{ for $1 \leq i<j \leq d$.}
\end{equation}
The result obtained in \cite[Theorem 1.1]{KoNi2021} can be 
phrased like this: if the tuple of weights 
$\uw = (w_1, \ldots , w_d)$ from Notation \ref{def:22} happens 
to have $w_1 = \cdots = w_d = 1/d$, then the limit law $\muw$
from Theorem \ref{thm:25} is precisely the law (a.k.a.~the average
empirical eigenvalue distribution) of the traceless GUE matrix 
$M$ from Equation (\ref{eqn:23x}).  The main point of the present 
paper is to indicate a nice phenomenon which appears when some of 
the weights $w_1, \ldots , w_d$ are allowed to be different from 
each other: 

\begin{equation}   \label{eqn:23y}
\begin{array}{lcl}
\vline  &  
\mbox{The limit law $\mu_{\uw}$ still is the law of a 
matrix $M$ like in (\ref{eqn:23x}).} &  \vline          \\
\vline  &  
\mbox{Only that now an off-diagonal entry $g_{i,j}$ and 
its adjoint $g_{j,i} = g_{i,j}^{*}$}  &  \vline         \\
\vline &
\mbox{are set to satisfy the relation
$g_{i,j} g_{j,i} = g_{j,i} g_{i,j} + (w_j - w_i)$.}  & \vline
\end{array}   
\end{equation}

Of course, if $i,j \in \{ 1, \ldots ,d \}$ are such that 
$w_i \neq w_j$, then the commutation relation stated in 
(\ref{eqn:23y}) forces $g_{i,j}$ to no longer be a random
variable in the usual sense.  In order to substantiate the 
statement made in (\ref{eqn:23y}), we thus introduce the 
following notion.

\vspace{6pt}

\begin{definition}   \label{def:26}
{\em (CCR-analogue of a complex Gaussian random variable.)}

\noindent
Let $( \cA , \varphi )$ be a $*$-probability space and let 
$\wx_{(1,*)}, \wx_{(*,1)}$ be two parameters in $(0, \infty )$.
We will say that an element $a \in \cA$ is
a {\em centred CCR-complex-Gaussian element} with parameters 
$\wx_{(1,*)}$ and $\wx_{(*,1)}$ when we have the commutation
relation
\begin{equation}   \label{eqn:26a}
a^{*} a = a a^{*} + ( \wx_{(*,1)} - \wx_{(1,*)}) \oneA,
\end{equation}
and we have the expectation formula
\begin{equation}   \label{eqn:26b}
\varphi \bigl( a^p \, (a^{*})^q \, \bigr) = 
\left\{  \begin{array}{ll}
0, & \mbox{ if $p \neq q$;}  \\
   &                            \\
p! \ \wx_{(1,*)}^p, & \mbox{ if $p = q$.}
\end{array}   \right\} , \mbox{ for $p,q \in \bN \cup \{ 0 \}$.}
\end{equation}
\end{definition}

\vspace{6pt}

\begin{remark}    \label{rem:27}
The term ``Gaussian'' used in Definition \ref{def:26} is 
justified by looking at the special case when
$\wx_{(1,*)} = \wx_{(*,1)} =: \wx$.  In such a case, 
Equation (\ref{eqn:26a}) says that $a$ 
commutes with $a^{*}$ (hence can be treated like a complex
random variable from classical probability), while
(\ref{eqn:26b}) becomes the formula giving the 
joint moments of $f$ and $\overline{f}$ for a centred 
complex Gaussian variable $f$ of variance $\omega$.

We note that the commutation in (\ref{eqn:26a}) together 
with the prescription from (\ref{eqn:26b}) determine all 
the joint moments of $a$ and $a^{*}$.  In particular one gets, 
symmetrically to the second branch of (\ref{eqn:26b}), that 
$\varphi \bigl( \, (a^{*})^p \, a^p  \bigr) = p! \ \wx_{(*,1)}^p$
for every $p \in \bN$.  This is a special case of a ``Wick-style''
formula for computing joint moments of $a$ and $a^{*}$ which is
given in Proposition \ref{prop:83} in the body of the paper.

We next introduce a CCR-analogue for the notion of traceless GUE
matrix.  Concerning the diagonal elements $a_{1,1}, \ldots , a_{d,d}$
appearing in the next definition, observe that the covariance matrix
$C$ given in (\ref{eqn:28b}) below is a generalization of
(\ref{eqn:23xx}) from the description of a traceless GUE.  The fact 
that $a_{1,1}, \ldots , a_{d,d}$ form a Gaussian family can be 
understood in the sense that their joint moments are given by a 
rather standard Wick formula (cf.~review in Definition \ref{def:82}.2 
below).
\end{remark}

\vspace{6pt}

\begin{definition}    \label{def:28}
{\em (CCR analogue for a traceless GUE matrix.)}
Let $( \cA , \varphi )$ be a $*$-probability space.  
Suppose we have a family of 
unital $*$-subalgebras of $\cA$, denoted as
\[
\{ \cA_o \} \cup \{ \cA_{i,j} \mid 1 \leq i < j \leq d \},
\]
which are commuting independent (a standard notion -- 
cf.~review in Definition \ref{def:82}.1 below). Suppose 
moreover that we have some elements of $\cA$, as follows.

\vspace{6pt}

\noindent
(i) For every $1 \leq i < j \leq d$,  we have a centred 
CCR-complex-Gaussian element $a_{i,j} \in \cA_{i,j}$ with parameters 
$w_j$ and $w_i$ (meaning that in Definition \ref{def:26} we take
$\wx_{(1,*)} = w_j$ and $\wx_{(*,1)} = w_i$).
We put $a_{j,i} := a_{i,j}^{*} \in \cA_{i,j}$, thus getting the 
relation $a_{j,i} a_{i,j} = a_{i,j} a_{j,i} + (w_i - w_j) \oneA$.

\vspace{6pt}

\noindent
(ii) $\cA_o$ is commutative and we have
selfadjoint elements $a_{1,1}, \ldots , a_{d,d} \in \cA_o$
which form a centred Gaussian family with covariance matrix
$C = [c_{i,j}]_{i,j=1}^d$, where:
\begin{equation}   \label{eqn:28b}
c_{i,i} = w_i - w_i^2
\mbox{ for $1 \leq i \leq d$, and }
c_{i,j} = c_{j,i} = - w_i w_j
\mbox{ for $1 \leq i < j \leq d$.}
\end{equation}

\vspace{6pt}

\noindent
Then the selfadjoint matrix 
$M = [ a_{i,j} ]_{1 \leq i,j \leq d}$ in 
$M_d ( \cA )$ is said to be a {\em traceless CCR-GUE
matrix with parameters $w_1, \ldots , w_d$}.
\end{definition}

\vspace{6pt}

The main result of the present paper is then 
stated as follows.

\vspace{6pt}

\begin{theorem}    \label{thm:29}
Let $( \cA , \varphi )$ be a $*$-probability space, and let 
$M = [ a_{i,j} ]_{i,j=1}^d \in M_d ( \cA )$ be a traceless 
CCR-GUE matrix with parameters $w_1, \ldots , w_d$, as in 
the preceding definition.  Consider the linear functional
$\varphi_{\uw} : M_d ( \cA ) \to \bC$ defined by
\begin{equation}   \label{eqn:29a}
\varphi_{\uw} ( X ) = \sum_{i=1}^d  w_i 
\, \varphi ( x_{i,i} ), \ \mbox{ for }
X = [ x_{i,j} ]_{i,j=1}^d \in M_d ( \cA ).
\end{equation}
Then the law of $M$ in the $*$-probability space 
$( M_d ( \cA ), \varphi_{\uw} )$ is equal to the limit law
$\mu_{\uw}$ from the theorem of CLT type discussed in
Section 2.2.
\end{theorem}

\vspace{6pt}

\begin{remark}     \label{rem:210}
In the special case when $w_1 = \cdots = w_d = 1/d$, the 
traceless CCR-GUE matrix $M$ from Definition \ref{def:28} 
becomes a usual traceless GUE matrix, as defined in 
Equation (\ref{eqn:23x}) at the beginning of this subsection.
Let us simply denote by $\mu_d$ the limit law $\muw$ which 
appears in this special case, and let us use the notation
$\nu_d$ for the law (a.k.a.~``average empirical eigenvalue 
distribution'') of the actual GUE matrix $G$.
It turns out that $\mu_d$ and $\nu_d$ are related by 
a nice convolution formula:
\begin{equation}    \label{eqn:210a}
\nu_d = \mu_d * N(0, 1/d^2),
\end{equation}
where $N(0, 1/d^2)$ is the centred normal distribution of variance
$1/d^2$.  In the paper \cite{KoNi2021}, the formula (\ref{eqn:210a}) 
was used in order to write explicitly the Laplace transform of the 
limit law $\mu_d$ (which is easily done, based on the fact that the 
Laplace transform of $\nu_d$ is well-known -- see e.g.~the 
Section 3.3 of the monograph \cite{AnGuZe2009}). 

The paper \cite{KoNi2021} is exclusively devoted to the case when
$w_1 = \cdots = w_d = 1/d$.  It is perhaps amusing
to note that the authors of that paper were not aware 
of the interpretation of $\mu_d$ as distribution of the traceless GUE,
and derived Equation (\ref{eqn:210a}) directly from the interpretation
of $\mu_d$ as limit law in the CLT for star-generators.  Once we know 
that $\mu_d$ is the distribution of the traceless GUE matrix $M$, it 
becomes of course much easier to obtain (\ref{eqn:210a}) via a 
calculation based on the relation $G = M + \frac{\zeta}{d} I_d$, which 
is a re-writing of the definition of $M$ in (\ref{eqn:23x}).
\end{remark}

$\ $

\subsection{Underlying combinatorics: two permutations associated 
to a \boldmath{$\pi \in \cP_2 (k)$}.}

$\ $

\noindent
In this subsection we point out an interesting combinatorial 
phenomenon which connects the Theorems 
\ref{thm:25} and \ref{thm:29} stated above.  The limit 
law $\muw$ turns out to be symmetric with finite moments of
all orders, and the proofs of both Theorems
\ref{thm:25} and \ref{thm:29} rely on explicit formulas 
for the even moments of $\muw$.  In both cases, these explicit 
formulas express the moment of order $k$ of $\muw$ as a 
summation over the set $\cP_2 (k)$ of pair-partitions of 
$\{ 1, \ldots , k \}$; and in both cases the general term of
the summation, indexed by a $\pi \in \cP_2 (k)$, involves a 
certain permutation in $S_{\infty}$ that is associated to $\pi$.
But then things go as follows.

\vspace{4pt}

\noindent
-- The moment formulas coming from Theorem \ref{thm:25}
use a permutation ``$\tau_{\pi}$'' associated 

to $\pi$, where $\tau_{\pi}$ is defined as a certain 
product of star-transpositions.

\vspace{4pt}

\noindent
-- The moment formulas needed in Theorem \ref{thm:29} are
Wick-style formulas, and use a 

permutation ``$\sigma_{\pi}$'' associated to $\pi$, where 
$\sigma_{\pi}$ is obtained by simply viewing every

pair of $\pi$ as a transposition.  

\vspace{4pt}

\noindent
At first sight, the permutations 
$\tau_{\pi}$ and $\sigma_{\pi}$ do not appear to be directly 
related to each other.  For a concrete example, say for instance
that $k=8$ and we are dealing with
\[
\pi = \bigl\{ \{ 3,8 \}, \{ 4,7 \}, \{ 1,6 \}, \{ 2,5 \} \bigr\}
\in \cP_2 (8).
\]
In order to be consistent with the notation used in the body of
the paper, we listed the pairs $V_1, \ldots , V_4$ of $\pi$ in 
decreasing order of their maximal elements: 
$V_1 = \{ 3,8 \}, \ldots , V_4 = \{ 2,5 \}$.  The permutation   
$\sigma_{\pi}$ simply is the product of the 4 disjoint transpositions
$(3,8), (4,7), (1,6)$ and $(2,5)$. The permutation $\tau_{\pi}$ is 
found via a more elaborate procedure, described as follows: on a 
picture of $\pi$ as shown in Figure 1 below, we draw $\gamma_1$'s 
on top of the two elements of $V_1$, we draw $\gamma_2$'s on top of
the two elements of $V_2$, and so on.  Then $\tau_{\pi}$ is defined
as the product of the 8 star-transpositions that were drawn in this
way, taken from left to right:
\[
\tau_{\pi} =  \gamma_3 \gamma_4 \gamma_1 \gamma_2 
\gamma_4 \gamma_3 \gamma_2 \gamma_1  
= (1,4)(1,5)(1,2)(1,3)(1,5)(1,4)(1,3)(1,2)
= (1,5,3) \in S_5 \subseteq S_{\infty}.
\]

\begin{center}
\setlength{\unitlength}{0.7cm}
$\pi = \
\begin{picture}(9,2) \thicklines
   \put(0,-1){\line(0,1){2}}
   \put(0,-1){\line(1,0){5}}
   \put(1,-0.5){\line(0,1){1.5}}
   \put(1,-0.5){\line(1,0){3}}
   \put(2,-1.5){\line(0,1){2.5}}
   \put(2,-1.5){\line(1,0){5}}
   \put(3,0){\line(0,1){1}}
   \put(3,0){\line(1,0){3}}
   \put(4,-0.5){\line(0,1){1.5}}
   \put(5,-1){\line(0,1){2}}
   \put(6,0){\line(0,1){1}}
   \put(7,-1.5){\line(0,1){2.5}}
   \put(-0.3,1.2){$\gamma_3$}
   \put(0.7,1.2){$\gamma_4$}
   \put(1.7,1.2){$\gamma_1$}
   \put(2.7,1.2){$\gamma_2$}
   \put(3.7,1.2){$\gamma_4$}
   \put(4.7,1.2){$\gamma_3$}
   \put(5.7,1.2){$\gamma_2$}
   \put(6.7,1.2){$\gamma_1$}
\end{picture}$

\vspace{1.5cm}

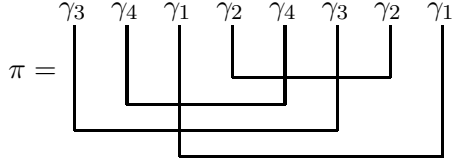
{\bf Figure 1.} {\em A picture used in the
computation of $\tau_{\pi}$, for a
$\pi \in \cP_2 (8)$.}
\end{center}

\vspace{10pt}

For a general $\pi \in \cP_2 (k)$, the product of 
star-transpositions which defines $\tau_{\pi}$ has 
2 occurrences of each of $\gamma_1, \ldots , \gamma_{k/2}$, 
thus cannot move any number $n > (k+2)/2$, and 
belongs to the subgroup $S_{(k+2)/2}$ of $S_{\infty}$.  
We are, in fact, only interested in the sizes of
orbits of $\tau_{\pi}$, as this is the information 
needed in order to compute $\chi ( \tau_{\pi} )$;
it is not obvious, though, how these sizes of orbits 
can be expressed in terms of the ``other'' permutation 
$\sigma_{\pi} \in S_k \subseteq S_{\infty}$. 
The nice fact we want to signal here is that the desired
orbit sizes can actually be read from the 
permutation $\cyc_{k+1} \cdot \sigma_{\pi} \in S_{k+1}$,
where $\cyc_{k+1} = (1,2, \ldots ,k,k+1)$.
More precisely, the following holds.

\vspace{6pt}

\begin{proposition}  \label{prop:211}
Let 
$\pi = \{ V_1, \ldots , V_{k/2} \} \in \cP_2 (k)$
and let us consider the permutations 
$\tau_{\pi} \in S_{(k+2)/2}$ and $\sigma_{\pi} \in S_k$,
as described above.  Consider moreover the decomposition
\[
\{ 1, \ldots , k+1 \} = R_1 \cup \cdots \cup R_p
\]
of $\{ 1, \ldots , k+1 \}$ into orbits
of $\cyc_{k+1} \cdot \sigma_{\pi}$.  Then the 
decomposition of $\{ 1, \ldots , (k+2)/2 \}$ into orbits 
of $\tau_{\pi}$ has $p$ orbits, of cardinalities
\[
| R_1 \cap B |, \ldots , | R_p \cap B |
\ \mbox{ (in some order),}
\]
where $B := 
\{ \max (V_i) \mid 1 \leq i \leq k/2 \} \cup \{ k+1 \}$.
\end{proposition}

\vspace{6pt}

In the concrete example shown in Figure 1: one has
\[
\cyc_9 \cdot \sigma_{\pi}
= (1,2, \ldots , 9) \cdot (3,8)(4,7)(1,6)(2,5)
= (1,7,5,3,9)(2,6)(4,8).
\]
Intersecting the orbits of $\cyc_9 \cdot \sigma_{\pi}$
with $B = \{ 5,6,7,8,9 \}$ splits $B$
into a $3+1+1$ partition, indeed of the 
same type as the partition of $\{ 1, \ldots , 5 \}$ into
orbits of $\tau_{\pi}$.

Proposition \ref{prop:211} (or rather, a slight generalization 
shown in Section 6 below, which allows $\pi$ to have 
singleton blocks) is the combinatorial link connecting 
Theorems \ref{thm:25} and \ref{thm:29}.

$\ $

\section{Background: the framework of $( \cM , \tr )$ and 
the law of large numbers}

In this section we clarify what is the non-commutative probability 
space we will work with, and we discuss the law of large numbers
which is a prerequisite for a theorem of CLT kind.
We will use the framework considered in Section 2.1: 
the tuple of weights $\uw = (w_1, \ldots , w_d)$ 
that is fixed for the whole paper, the symmetric power sums 
$1 = \symp_1 > \symp_2 > \cdots > \symp_n > \cdots$ 
from Equation (\ref{eqn:22b}), the character 
$\chi : S_{\infty} \to \bR$ defined by using the $\symp_n$'s
in Equation (\ref{eqn:22c}), and the GNS representation 
$U : S_{\infty} \to B( \cH )$ of the character $\chi$, which was
introduced in Notation \ref{def:23}.

$\ $

\subsection{The \boldmath{$W^{*}$}-probability space 
\boldmath{$( \cM , \tr )$}.}

\begin{remark}   \label{rem:31}
$1^o$ Recall that we use the 
notation $\widehat{\sigma}$ for the vector in $\cH$ corresponding 
to a permutation $\sigma \in S_{\infty}$.  From the prescription 
of inner products stated in Notation \ref{def:23} it is clear 
that $\| \widehat{\sigma} \| = 1$ for every 
$\sigma \in S_{\infty}$, while for any $\sigma \neq \tau$ in 
$S_{\infty}$, the common value of
$\langle \widehat{\sigma}, \widehat{\tau} \rangle$ and
$\langle \widehat{\tau}, \widehat{\sigma} \rangle$ falls 
in the interval $(0,1)$.  This also has the immediate consequence
that the map $S_{\infty} \ni \sigma \mapsto \widehat{\sigma} \in \cH$
is injective, since
\[
\| \widehat{\sigma} - \widehat{\tau} \|^2 
= 2 \bigl( 1 - \vvarphi ( \sigma \tau^{-1} ) \, \bigr) > 0, 
\ \ \forall  \, \sigma \neq \tau 
\mbox{ in } S_{\infty}.
\]

\vspace{3pt}

\noindent
$2^o$ In general, we cannot count on the set 
$\{ \widehat{\sigma} \mid \sigma \in S_{\infty} \} \subseteq \cH$
to be linearly independent.  But we can at least be sure that
the sequence $( \,  \widehat{\gamma_n} \, )_{n=1}^{\infty}$ is
linearly independent.  This is because, for every $n \geq 1$, 
the Gram matrix associated to the vectors
$\widehat{\gamma_1}, \ldots , \widehat{\gamma_n}$ has entries
\[
\langle \widehat{\gamma_i} , \widehat{\gamma_j} \rangle
= \vvarphi ( \gamma_i \gamma_j )
= \left\{  \begin{array}{ll}
1, & \mbox{ if $i=j$,}   \\
\llambda_3, & \mbox{ if $i \neq j$;}
\end{array}   \right.
\]
this Gram matrix is then found to be invertible, 
since $0 < \llambda_3 < 1$.
\end{remark}

\vspace{6pt}

\begin{notation-and-remark}   \label{def:32}
$1^o$ It is immediate that the linear span of the operators 
$U( \sigma )$ is a unital $*$-subalgebra of $B( \cH )$ and, 
consequently, that the closure  
\begin{equation}   \label{eqn:32a}
\cM : = \overline{\linspan}^{\WOT} 
\{ U(\sigma) : \sigma \in S_{\infty} \} \subseteq B(\cH)
\end{equation}
is a von Neumann algebra of operators on $\cH$.  

\vspace{3pt}

\noindent
$2^o$ Consider the special vector $\widehat{\cyc_1} \in \cH$
(where recall that ``$\cyc_1$'' stands for the unit of $S_{\infty}$).
Then $\widehat{\cyc_1}$ is a trace-vector for $\cM$; that is,
the vector-state
\begin{equation}   \label{eqn:32b}
\tr : \cM \to \bC, \ \ \tr (T) :=
\langle \, T ( \widehat{\cyc_1} ), \widehat{\cyc_1} \, \rangle
\mbox{ for } T \in \cM
\end{equation}
is a {\em trace}.  The trace property of $\tr$ is an immediate
consequence of the fact that $\chi$ is a class-function: 
one has 
$\tr \bigl( \, U( \sigma ) U ( \tau ) \, \bigr)
= \chi ( \sigma \tau ) = \chi ( \tau \sigma )
= \tr \bigl( \, U( \tau ) U ( \sigma ) \, \bigr)$
for every $\sigma , \tau \in S_{\infty}$,
which then implies that $\tr (AB) = \tr (BA)$ for all 
$A,B \in \cM$.  

\vspace{3pt}

\noindent
$3^o$ Standard arguments using the representation
$U' : S_{\infty} \to B( \cH )$ defined by multiplication on 
the right ($[ U'( \sigma ) \bigl] \, ( \widehat{\tau} ) 
= \widehat{ \tau  \sigma^{-1}}$
for $\sigma, \tau \in S_{\infty}$)
yield the fact that the linear map
\begin{equation}   \label{eqn:32c}
\cM \ni T \mapsto T ( \, \widehat{\cyc_1} \, ) \in \cH
\end{equation}
is injective. 
Hence putting
$\| T \|_2 := \| \, T( \widehat{\cyc_1} ) \, \|$ for $T \in \cM$
defines a norm on $\cM$.  A known property of this norm
(see e.g. \cite[Section II.2]{Takesaki}) is that it metrizes the 
SOT-topology on the 
unit ball $\cB := \{ T \in \cM \mid \ \|T\| \leq 1 \}$.  So,
in particular: if $T$ and $( T_n )_{n=1}^{\infty}$ are in $\cB$,
then verifying that $\limn \| T_n - T \|_2 = 0$ will be 
sufficient to ensure that $( T_n )_{n=1}^{\infty}$ is 
SOT-convergent to $T$.

\vspace{3pt}

\noindent
$4^o$  As a consequence of the injectivity of the map 
(\ref{eqn:32c}), the trace-state $\tr$ is
{\em faithful}, i.e.~has the property that 
$\tr (T^{*}T) > 0$ for every $T \neq 0$ in $\cM$.

In summary: we see that $\tr$ is a WOT-continuous 
faithful trace-state of the von Neumann algebra $\cM$.  
The couple $( \cM , \tr )$ is what 
one refers to as a {\em tracial $W^{*}$-probability space};
this $( \cM , \tr )$ is the framework we will use throughout the paper.
\end{notation-and-remark}

$\ $

\subsection{The operator \boldmath{$A_0 \in \cM$}
and the law of large numbers for \boldmath{$U( \gamma_n )$}'s.}

$\ $

\noindent
We now start looking at the sequence of operators 
$U( \gamma_n )$ in $\cM$.  Observe that the $U( \gamma_n )$'s 
are linearly independent; this is clear from the fact that, 
upon applying these operators to the vector $\widehat{\cyc_1}$,
we get the linearly independent sequence
of vectors $( \widehat{\gamma_n} )_{n=1}^{\infty}$ of $\cH$.
It turns out that the sequence 
$\bigl( \, U( \gamma_n ) \, \bigr)_{n=1}^{\infty}$ is
WOT-convergent.  The limit is a special operator $A_0 \in \cM$
used in \cite{GoKo2010}, which played a crucial role in the 
considerations of non-commutative dynamics made in that paper. 
For the sake of keeping the presentation
self-contained, we include below the relatively easy proofs of the 
facts we need, in particular (cf.~Proposition \ref{prop:34}) of the 
occurrence of $A_0$ in the law of large numbers we are interested in. 

\vspace{6pt}

\begin{proposition}   \label{prop:33}
The sequence $\bigl( U( \gamma_n ) \bigr)_{n=1}^{\infty}$ has 
a WOT-limit $A_0 \in \cM$, where $A_0 = A_0^{*}$ and $\| A_0 \| \leq 1$.
The operator $A_0$ can be described via its action on vectors, 
as follows:
\begin{equation}   \label{eqn:33a}
\left\{   \begin{array}{c}
\mbox{ For every $\sigma, \tau \in S_{\infty}$, one has } 
\langle  \, A_0 ( \widehat{\sigma} ), \widehat{\tau} \, \rangle 
= \frac{ \llambda_{1+|V|} }{ \llambda_{|V|} } \,
\langle \, \widehat{\sigma}, \widehat{\tau} \, \rangle,     \\  \vspace{4pt}
\mbox{ where $V$ is the orbit of $\sigma \tau^{-1}$ which 
contains the number $1$.}
\end{array}   \right.
\end{equation}
\end{proposition}

\begin{proof}   We start by verifying a formula related to 
(\ref{eqn:33a}), stated as follows: let $\sigma, \tau$ be in 
$S_{\infty}$, let $k \geq 1$
be such that $\sigma , \tau \in S_k$ (i.e.~such that 
$\sigma (n) = \tau (n) = n$ for all $n > k$), and let $V$ be the
orbit of $\sigma \tau^{-1}$ which contains the number $1$.  Then  
\begin{equation}   \label{eqn:33b}
\langle  \, [U( \gamma_n )] ( \widehat{\sigma} ), \widehat{\tau} \, \rangle 
= \frac{ \llambda_{1+|V|} }{ \llambda_{|V|} } \,
\langle \, \widehat{\sigma}, \widehat{\tau} \, \rangle, 
\ \ \forall \, n \geq k.
\end{equation}
Indeed, the equality stated in (\ref{eqn:33b}) 
amounts to 
\[
\vvarphi ( \gamma_n \sigma \tau^{-1} ) 
= \frac{ \llambda_{1+|V|} }{ \llambda_{|V|} } \,
\vvarphi ( \sigma \tau^{-1} ), \ \ \forall \, n \geq k.
\]
This follows directly from the formula defining $\vvarphi$ in 
(\ref{eqn:22c}), combined with the following observation 
about orbit structure: when $n \geq k$, 
the permutation $\gamma_n \sigma \tau^{-1}$ has the same orbits 
as $\sigma \tau^{-1}$, with the only difference that the
number $n+1$ is now inserted into the orbit $V$ of $\sigma \tau^{-1}$,
right before the occurrence of $1$ in that orbit.  (Denoting 
$( \tau \sigma^{-1} ) (1) := h \in \{ 1, \ldots , k \}$, one 
has that $\sigma \tau^{-1}$ sends $h$ directly to $1$, while 
$\gamma_n \sigma \tau^{-1}$ sends $h \mapsto n+1 \mapsto 1$.)

From (\ref{eqn:33b}) it easily follows, by approximating
arbitrary $\xi_1, \xi_2 \in \cH$ with linear combinations of vectors 
$\widehat{\sigma}, \widehat{\tau}$, that the limit 
\begin{equation}   \label{eqn:33c}
\beta ( \xi_1 , \xi_2 ) := \lim_{n \to \infty} 
\langle  \, [U( \gamma_n )] ( \xi_1 ), \xi_2 \, \rangle 
\end{equation}
exists for every $\xi_1, \xi_2 \in \cH$.  Equation (\ref{eqn:33c})
defines a sesquilinear functional $\beta$ on $\cH$, with the 
property that 
$| \, \beta ( \xi_1, \xi_2 ) \, | \leq \| \xi_1 \| \cdot \| \xi_2 \|$
for all $\xi_1, \xi_2 \in \cH$, and a standard argument infers from 
here the existence of an $A_0 \in B( \cH )$ such that 
$\beta ( \xi_1 , \xi_2 ) = \langle A_0 ( \xi_1 ) , \xi_2 \rangle$ for 
all $\xi_1 , \xi_2 \in \cH$.  In combination with (\ref{eqn:33c}),
the latter formula says
that the sequence $( U( \gamma_n ) )_{n=1}^{\infty}$ converges WOT to 
$A_0$.  We have that $A_0 \in \cM$ with $A_0^{*} = A_0$ and 
$\| A_0 \| \leq 1$, because the $U( \gamma_n )$'s have these properties.
Finally, making $n \to \infty$ in Equation (\ref{eqn:33b}) gives the 
formula (\ref{eqn:33a}).
\end{proof}

\vspace{6pt}

\begin{proposition}   \label{prop:34}
{\em (Law of large numbers.)}  
Let $A_0 \in \cM$ be as in  
Proposition \ref{prop:33}.  Then
\begin{equation}   \label{eqn:34a}
\mathrm{SOT}-\lim_{n \to \infty} \frac{1}{n}
\sum_{i=1}^n U(\gamma_i) = A_0.
\end{equation}
\end{proposition}

\begin{proof} We will prove that 
\begin{equation} \label{eqn:34b}
\left\| \, \frac{1}{n} 
\sum_{i=1}^n U(\gamma_i) - A_0 \, \right\|_2^2 
= \frac{1 - \llambda_3}{n}, \ \ \forall \, n \geq 1.
\end{equation}
Making $n \to \infty$ in (\ref{eqn:34b}) and 
using the fact about SOT-convergence that was noted in 
Remark \ref{def:32}.3 will then yield the 
statement of the proposition.

The squared $\| \cdot \|_2$-norm indicated on the left-hand side of 
(\ref{eqn:34b}) is equal, by definition, to the squared norm of 
the vector $\frac{1}{n}  \sum_{i=1}^n  
\widehat{\gamma_i} - A_0 ( \widehat{\cyc_1} ) \in \cH$,
which is in turn equal to:
\begin{equation}   \label{eqn:34c}
\frac{1}{n^2} \left\| \sum_{i=1}^n \widehat{\gamma_i} \right\|_2^2
- 2 \mathrm{Re} \, \left\la \frac{1}{n} \sum_{i=1}^n 
\widehat{\gamma_i} , A_0 ( \widehat{\cyc_1} ) \right\ra 
+ \| \, A_0 ( \widehat{\cyc_1} ) \, \|^2.
\end{equation}
Let us then fix an $n \in \bN$ and compute the quantities 
indicated in (\ref{eqn:34c}).  First,
\[ \left\| \sum_{i=1}^n \widehat{\gamma_i} \right\|_2^2
= \sum_{i,j=1}^n \langle 
\, \widehat{\gamma_i}, \widehat{\gamma_j} \, \rangle
= \sum_i \vvarphi ( \gamma_i^2 ) 
+ \sum_{i \neq j} \vvarphi ( \gamma_i \gamma_j )
= n \cdot 1 + (n^2 -n) \llambda_3 \text{.} \]
Then we note that
\begin{equation}   \label{eqn:34d} 
\langle \, A_0 ( \widehat{\cyc_1} ), 
\widehat{\gamma_i} \, \rangle
= \lim_{m \to \infty} \langle \, 
[ U ( \gamma_m ) ] \, ( \widehat{\cyc_1} ), 
 \widehat{\gamma_i} \, \rangle
 = \lim_{m \to \infty}
 \vvarphi ( \gamma_m \gamma_i ) = \llambda_3,
 \ \ \forall \, i \in \bN.
 \end{equation}
Finally, for $\| \, A_0 ( \widehat{\cyc_1} ) \, \|^2$
we compute:
\[
\langle \, A_0 ( \widehat{\cyc_1} ),
A_0 ( \widehat{\cyc_1} ) \, \rangle 
= \lim_{m \to \infty} 
\langle \, [ U( \gamma_m ) ] \, ( \widehat{\cyc_1} ),
A_0 ( \widehat{\cyc_1} ) \, \rangle
= \lim_{m \to \infty} 
\langle \, \widehat{\gamma_m} ,
A_0 ( \widehat{\cyc_1} ) \, \rangle  
= \llambda_3,
\]
where at the first equality sign we invoked the 
fact that $A_0$ is the WOT-limit of the $U( \gamma_m )$'s,
and at the third equality sign we made use of
(\ref{eqn:34d}).  Upon putting all these things together
and plugging them into (\ref{eqn:34c}), we find that the 
quantity appearing in (\ref{eqn:34c}) is
indeed equal to $(1 - \llambda_3)/n$,
as claimed in (\ref{eqn:34b}).
\end{proof}

\vspace{6pt}

The calculation of $\| \, A_0 ( \widehat{\cyc_1} ) \, \|^2$
(equivalently, of $\tr ( A_0^2 )$) shown above gives a special case
of a formula for moments with respect to the trace, which we record
in the next proposition.

\vspace{6pt}

\begin{proposition}   \label{prop:35} 
Consider a tuple 
$\ui : \{ 1, \ldots ,k \} \to \bN \cup \{ \infty \}$
and the operators $T_1, \ldots , T_k \in \cM$ defined by
\begin{equation}  \label{eqn:35a}
T_h = \left\{  \begin{array}{ll}
U( \gamma_{\ui (h)} ), &  \mbox{ if $\ui (h) \in \bN$,} \\
A_0,                   &  \mbox{ if $\ui (h) = \infty$}
\end{array}   \right\} , \ \ 1 \leq h \leq k.
\end{equation}
Let $\uj : \{ 1, \ldots , k \} \to \bN$ be a tuple obtained out 
of $\ui$ by replacing its infinite values by some ``new, distinct 
finite values''.  That is, $\uj$ is obtained by using the following 
procedure:
\begin{equation}   \label{eqn:35b}
\begin{array}{ll}
\vline  & 
\mbox{-- let $Q := \{ q \in \bN \mid \exists \, 1 \leq h \leq k$
such that $\ui (h) = q \}$;    }  \\ 
\vline  &                         \\     
\vline  &
\mbox{-- choose an injective function 
   $f : \ui^{-1} ( \infty )  \to \bN \setminus Q$; }          \\ 
\vline  &                          \\
\vline  &   
\mbox{-- define $\uj : \{ 1, \ldots , k \} \to \bN$ by putting }
\uj (h) = \left\{  \begin{array}{ll}
\ui (h), &  \mbox{ if $\ui (h) \in \bN$,} \\
f(h),    &  \mbox{ if $\ui (h) = \infty$}
\end{array}  \right\} .
\end{array}
\end{equation}
Then one has that
\begin{equation}   \label{eqn:35c}
\tr (T_1 \cdots T_k) = \tr \Bigl( 
\, U( \gamma_{\uj (1)} ) \cdots U( \gamma_{\uj (k)} ) \, \Bigr)
= \vvarphi ( \gamma_{\uj (1)} \cdots \gamma_{\uj (k)} ).
\end{equation}
\end{proposition}

\begin{proof} We proceed by induction on the cardinality
$\ell$ of $\ui^{-1} ( \infty )$.
The base case of the induction is $\ell = 0$; in this case
(\ref{eqn:35c}) holds trivially, since the procedure described
in (\ref{eqn:35b}) gives $\uj = \ui$.  We thus focus
on the induction step: we fix an $\ell \geq 1$, we assume the 
statement of the proposition holds for $\ell -1$, and we will 
prove that it holds for $\ell$ as well.

So let $\ui : \{ 1, \ldots , k \} \to \bN \cup \{ \infty \}$ with 
$| \ui^{-1}(\infty) | = \ell$,
and let $T_1, \ldots , T_k \in \cM$ be as 
in (\ref{eqn:35a}). 
We consider a tuple $\uj$ defined as in
(\ref{eqn:35b}), and we aim to prove 
that (\ref{eqn:35c}) holds.

We choose an $h_o \in \ui^{-1} ( \infty ) \subseteq \{ 1, \ldots , k \}$
and an $n_o \in \bN$ such that $n_o > \max (Q)$ and 
$n_o > \max \{ f(h) \mid h \in \ui^{-1} ( \infty ) \}$, where $f$ is the 
function used in (\ref{eqn:35b}) in order to construct the tuple $\uj$.
For every $n \geq n_o$ we consider the tuples 
$\ui^{(n)} : \{ 1, \ldots , k \} \to \bN \cup \{ \infty \}$ and
$\uj^{(n)} : \{ 1, \ldots , k \} \to \bN$ defined by 
\begin{equation}   \label{eqn:35d}
\ui^{(n)} (h) := \left\{ \begin{array}{ll}
\ui (h), & \mbox{ if $h \neq h_o$,}  \\
n, & \mbox{ if $h = h_o$}
\end{array}  \right\}, \ \mbox{ and }
\ \uj^{(n)} (h) := \left\{ \begin{array}{ll}
\uj (h), & \mbox{ if $h \neq h_o$,}  \\
n, & \mbox{ if $h = h_o$}
\end{array}  \right\} .
\end{equation}
For such $n$ we also consider the operators
$T_1^{(n)}, \ldots , T_k^{(n)} \in \cM$ defined by the same 
recipe as in Equation (\ref{eqn:35a}), but by starting from the 
tuple $\ui^{(n)}$ instead of $\ui$.  

For every $n \geq n_o$, the pre-image
$\bigl( \ui^{(n)} \bigr)^{-1} ( \infty )$ has cardinality 
$\ell -1$, hence the induction hypothesis applies to 
$\ui^{(n)}$.  It is immediate that, when used on 
$\ui^{(n)}$, the procedure described in (\ref{eqn:35b}) 
can be arranged to lead to the tuple $\uj^{(n)}$.  We thus
obtain that 
\begin{equation}  \label{eqn:35e}
\tr \bigl( T_1^{(n)} \cdots  T_k^{(n)} \bigr)
= \vvarphi \bigl( \gamma_{\uj^{(n)}(1)} 
\cdots \gamma_{\uj^{(n)}(k)} \bigr), \ \ n \geq n_o.
\end{equation}

We leave it as an exercise to the reader to check that,
for every $n \geq n_o$:
due to how $\uj^{(n)}$ is defined in (\ref{eqn:35d}) (and due to 
the specifics of how $\uj$ is obtained), the permutation
$\gamma_{\uj^{(n)}(1)} \cdots \gamma_{\uj^{(n)} (k)}$ 
has the same orbit structure as  
$\gamma_{\uj (1)} \cdots \gamma_{\uj (k)}$.  The character 
$\chi$ thus takes the same value on these two permutations.
We come to the conclusion that the right-hand side of
(\ref{eqn:35e}) is actually independent of $n$, and constantly equal 
to $\vvarphi ( \gamma_{\uj (1)} \cdots \gamma_{\uj (k)}$ ). 

Now to the left-hand side of (\ref{eqn:35e}): we observe that the operators 
indicated there can also be described as 
\begin{equation}   \label{eqn:35f}
T_h^{(n)} = \left\{   \begin{array}{ll}
T_h, & \mbox{ if $h \neq h_o$,}   \\
U ( \gamma_n ), & \mbox{ if $h = h_o$.}
\end{array}  \right.
\end{equation}
Thus $T_h^{(n)}$ differs from $T_h$ only on position $h = h_o$,
where we have
\[
\mathrm{WOT}-\lim_{n \to \infty} T_{h_o}^{(n)} 
= \mathrm{WOT}-\lim_{n \to \infty} U ( \gamma_n )
= A_0 = T_{h_o}.
\]
This further implies that
$\mathrm{WOT}-\lim_{n \to \infty}
T_1^{(n)} \cdots  T_k^{(n)} = T_1 \cdots T_k$, and
upon applying the WOT-continuous functional 
$\tr$ to the latter limit we find that 
\[
\lim_{n \to \infty} \tr \bigl(
T_1^{(n)} \cdots  T_k^{(n)} \bigr)  
= \tr ( T_1 \cdots T_k ).
\]
We conclude that making $n \to \infty$
in (\ref{eqn:35e}) leads to the Equation (\ref{eqn:35c}) we were after.
\end{proof}

\vspace{6pt}

\begin{corollary}   \label{cor:36}
$1^o$ For every $k \in \bN$, one has that 
$\tr ( A_0^k ) = \llambda_{+1}$. 

\vspace{3pt}

\noindent
$2^o$ The scalar spectral measure of $A_0$ with respect 
to $\tr$ is equal to $\sum_{i=1}^d w_i \delta_{w_i}$ 
(convex combination of Dirac measures).

\vspace{3pt}

\noindent
$3^o$ The spectrum of $A_0$ is equal to 
$\{ t \in (0,1) \mid \, \exists \, 1 \leq i \leq d
\mbox{ such that } w_i = t \}$.
\end{corollary}

\begin{proof} $1^o$ This is the special case of 
Proposition \ref{prop:35} when the tuple $\ui$ under consideration
has $\ui (h) = \infty$ for all $1 \leq h \leq k$.  Indeed, in this
special case we can pick the tuple $\uj : \{ 1, \ldots, k \} \to \bN$
to be defined by $\uj (h) = h$, $1 \leq h \leq k$, and 
Equation (\ref{eqn:35c}) simply tells us that
$\tr ( A_0^k ) = \vvarphi ( \gamma_1 \cdots \gamma_k )$.
Since $\gamma_1 \cdots \gamma_k = (k+1, k, \ldots , 2,1)$, a cycle of
length $k+1$, it follows that $\tr (A_0^k ) = \llambda_{k+1}$, as claimed.

\vspace{3pt}

\noindent
$2^o$ Let $\mu$ be the scalar spectral measure of $A_0$ with 
respect to $\tr$, and let $\mu ' := \sum_{i=1}^d w_i \delta_{w_i}$.
For every $k \geq 1$ we have
\[
\int_{\bR} t^k d \, \mu' (t) = \sum_{i=1}^d w_i \cdot (w_i)^k
= \sum_{i=1}^d w_i^{k+1} = \llambda_{k+1} = \tr (A_0^k )
= \int_{\bR} t^k \, d \mu (t),
\]
where at the last two equality signs we used the result of $1^o$ above
and the definition of a scalar spectral measure.  Thus $\mu$ and $\mu'$ 
are two compactly supported probability measures which have the same 
moments, and this implies that $\mu = \mu'$.

\vspace{3pt}

\noindent
$3^o$  Since $\tr$ is faithful, the spectrum of $A_0$ is known (see 
e.g.~\cite[Proposition 3.15]{NiSp2006}) to be equal to the support 
$\mathrm{supp} (\mu)$, where $\mu$ is the scalar spectral measure from
the proof of part $2^o$.  But from $2^o$ it is clear that 
$\mathrm{supp} ( \mu ) = 
\{ t \in (0,1) \mid \, \exists \, 1 \leq i \leq d$ such that $w_i = t \}$.
\end{proof}

$\ $

\section{The theorem of CLT type for the sequence of operators
$U( \gamma_n ) - A_0$}

Now that we identified in Proposition \ref{prop:34}
the law of large numbers for the sequence of represented 
star-generators $U( \gamma_n )$, we go to the next step 
usually taken towards a theorem of CLT type: we subtract the
limit $A_0$ out of the $U( \gamma_n )$'s, and we seek to find 
a limit in law for the rescaled averages
\[ 
\frac{1}{\sqrt{n}} \sum_{i=1}^n (U(\gamma_i) - A_0) \text{.} 
\]
This is covered by \cite[Theorem 9.4]{Ko2010}, which applies to sequences
with certain weaker probabilistic symmetries. For the reader's convenience,
we will give a self-contained presentation, tailored to the specific 
instance of CLT that we want to address.

$\ $

\subsection{Review of the CLT for exchangeable sequences.}

\begin{notation}   \label{def:41}
{\em (Review of some basic notation about set-partitions.)}
Let $k$ be in $\bN$.

\vspace{3pt}

\noindent
$1^o$ $\cP (k)$ will denote the set of all partitions of 
$\{ 1, \ldots , k \}$.  A partition $\pi \in \cP (k)$ is thus 
of the form $\pi = \{ V_1, \ldots , V_{\ell} \}$ where 
$V_1, \ldots , V_{\ell}$ (the {\em blocks} of $\pi$) are non-empty
pairwise disjoint sets with 
$V_1 \cup \cdots \cup V_{\ell} = \{ 1, \ldots , k \}$.

\vspace{3pt}

\noindent
$2^o$ On $\cP (k)$ we consider the partial order given by
{\em reverse refinement}: for $\pi, \rho \in \cP(k)$ we write  
``$\pi \leq \rho$'' to mean that every block of $\pi$
is contained in some block of $\rho$.  

\vspace{3pt}

\noindent
$3^o$ We denote $\cP_2 (k) := \{  \pi \in \cP (k) 
\mid \mbox{every block $V \in \pi$ has $|V| = 2$} \}$ 
(where $\cP_2 (k) = \emptyset$ if $k$ is odd).  The elements
of $\cP_2 (k)$ are called {\em pair-partitions}.

\vspace{3pt}

\noindent
$4^o$ We will treat a tuple $\ui \in \bN^k$ as a function 
$\ui : \{ 1, \ldots , k \} \to \bN$.  One defines the 
{\em kernel} of such a tuple as the partition 
$\Ker ( \pi ) \in \cP (k)$ defined as follows: two
numbers $p,q \in \{ 1, \ldots , k \}$ belong 
to the same block of $\Ker (\ui)$ if and only if 
$\ui(p) = \ui(q)$.
\end{notation}

\vspace{6pt}

\begin{notation-and-remark}   \label{def:42}
{\em (Exchangeable sequences.)}
Let $( \cA , \varphi )$ be a $*$-probability space and let 
$\ans$ be a sequence of selfadjoint elements of $\cA$.
Quantities of the form 
\[
\varphi (a_{\ui (1)} \cdots a_{\ui (k)} ), 
\mbox{ with $k \in \bN$ and $\ui  : \{ 1, \ldots , k \} \to \bN$}
\]
go under the name of {\em joint moments} of $\ans$.  We say that
$\ans$ is {\em exchangeable} to mean that its joint 
moments are invariant under the natural action of $S_{\infty}$,  
that is:
\begin{equation}   \label{eqn:42a}
\left\{   \begin{array}{c}
\varphi (a_{\ui (1)} \cdots a_{ \ui (k)} ) 
         = \varphi (a_{\uj (1)} \cdots a_{\uj (k)} )      \\
\mbox{for every $k \in \bN$ and 
      $\ui , \uj  : \{ 1, \ldots , k \} \to \bN$}         \\
\mbox{for which $\exists \, \sigma \in S_{\infty}$
      such that $\uj  = \sigma \circ \ui$.}
\end{array}   \right.
\end{equation}
It is easily seen that for two tuples
$\ui,\uj : \{ 1, \ldots , k \} \to \bN$, the existence of 
a $\sigma \in S_{\infty}$ such that 
$\uj = \sigma \circ \ui$ is equivalent to the fact that
$\Ker (\ui) = \Ker (\uj)$.  Thus the condition 
(\ref{eqn:42a}) could be equivalently written as
\begin{equation}   \label{eqn:42b}
\left\{   \begin{array}{c}
\varphi (a_{\ui (1)} \cdots a_{ \ui (k)} ) 
         = \varphi (a_{\uj (1)} \cdots a_{\uj (k)} )      \\
\mbox{for every $k \in \bN$ and $\ui , \uj  \in \bN^k $
      such that $\Ker ( \ui ) = \Ker ( \uj )$.}
\end{array}   \right.
\end{equation}
\end{notation-and-remark}

\vspace{6pt}

\begin{definition}   \label{def:43}
Let $( \cA , \varphi )$ be a $*$-probability space and let 
$\ans$ be an exchangeable sequence of selfadjoint elements 
of $\cA$.  

\vspace{6pt}

\noindent
$1^o$ The {\em function on partitions} associated to $\ans$ is
$\bt : \sqcup_{k=1}^{\infty} \cP (k) \to \bC$ defined as follows:
for every $k \in \bN$ and $\pi \in \cP (k)$ we put 
\begin{equation}   \label{eqn:43a}
\left\{  \begin{array}{l}
\bt ( \pi ) := \varphi \bigl(
a_{\ui (1)} \cdots a_{\ui (k)}  \bigr),
\mbox{ where $\ui \in \bN^k$ is}             \\
\mbox{ any $k$-tuple such that $\ker ( \ui ) = \pi$. }
\end{array}    \right.
\end{equation}
This formula is unambiguous due to (\ref{eqn:42b}) above.

\vspace{3pt}

\noindent
$2^o$ The sequence $\ans$ is said to have the
{\em singleton vanishing property} when the function $\bt$ 
defined in (\ref{eqn:43a}) satisfies:
\begin{equation}   \label{eqn:43b}
\left\{  \begin{array}{l}
\bt ( \pi ) = 0 \ \mbox{ whenever the partition 
$\pi \in \sqcup_{k=1}^{\infty} \cP (k)$ }                \\
\mbox{ has at least one block $V$ with $|V| = 1$.}
\end{array}  \right.
\end{equation}
Note that the singleton vanishing property guarantees centering: 
in the special case when $\pi$ is the unique partition in $\cP (1)$, 
(\ref{eqn:43b}) says that $\varphi (a_n) = 0$ for every $n \in \bN$.
\end{definition}

\vspace{6pt}

\begin{theorem}   \label{thm:44}
(CLT for an exchangeable sequence, following \cite{BoSp1996}.)

\noindent
Let $( \cA , \varphi )$ be a $*$-probability space and let
$\ans$ be a sequence of selfadjoint elements of $\cA$ which 
is exchangeable and has the singleton vanishing property.
Let $\bt : \sqcup_{k=1}^{\infty} \cP (k) \to \bC$ be the 
function on partitions associated to $\ans$ in Definition
\ref{def:43}.  Consider the linear functional
$\mu : \bC [X] \to \bC$ defined via the prescription that
$\mu (1) = 1$ and that 
\begin{equation}   \label{eqn:44a}
\mu ( X^k ) =
\sum_{ \rho \in \cP_2 (k)} \ \bt ( \rho ),
\ \ \forall \, k \in \bN ,
\end{equation}
with the empty sum on the right-hand side of (\ref{eqn:44a})
being read as $0$ for $k$ odd.  Then:

\vspace{6pt}

\noindent
$1^o$ $\mu$ is positive (that is, 
$\mu ( \, P \cdot \overline{P} \, ) \geq 0$ for every 
$P \in \bC [X]$).

\vspace{3pt}

\noindent
$2^o$ For every $n \in \bN$ put
\begin{equation}   \label{eqn:44b}
s_n := \frac{1}{\sqrt{n}} 
\bigl( a_1 + \cdots + a_n \bigr) \in \cA .
\end{equation}
Then $( s_n )_{n=1}^{\infty}$ converges in 
moments to $\mu$, that is, one has 
$\lim_{n \to \infty} \varphi (s_n^k) 
= \mu (X^k)$ for every $k \in \bN$.
\hfill $\square$
\end{theorem}

$\ $

\subsection{Back to our framework.}

$\ $

\begin{notation-and-remark}    \label{def:45}
We now return to the framework of the $W^{*}$-probability space 
$( \cM , \tr )$ from Section 3.1.  We use the same notation as
there, and we put:
\begin{equation}   \label{eqn:45a}
U_n := U(\gamma_n) \ \mbox{ and } 
\ \diamondU_n := U_n - A_0, \ \ n \in \bN .
\end{equation}
Our plan is to prove that
$\bigl( \diamondU_n \bigr)_{n=1}^{\infty}$
is exchangeable and has the singleton vanishing property, and 
then invoke the general Theorem \ref{thm:44} reviewed above.  

As explained in Section 2.2, the notation ``$\diamondU_n$'' goes 
in the same spirit as the commonly used notation
$\overset{\circ}{U}_n := U_n - \tr (U_n) \oneM 
= U_n - \llambda_2 \oneM$.
The element $\diamondU_n$ has in particular the centering property:
$\tr ( \diamondU_n ) = \tr ( U_n ) - \tr (A_0)
= \llambda_2 -\llambda_2 = 0$.
However, unless we are in the special case that our weights are 
$w_1 = \cdots = w_d = 1/d$, one has 
$\diamondU_n \neq \overset{\circ}{U}_n$.  This distinction is 
important, as the development shown below (more precisely, the 
verification of the singleton vanishing property) would not work
in connection to the elements $\overset{\circ}{U}_n$.

\vspace{3pt}

In order to verify the desired properties of
$\bigl( \diamondU_n \bigr)_{n=1}^{\infty}$,
it is convenient that we first introduce another bit of terminology 
concerning set-partitions.
\end{notation-and-remark}

\vspace{3pt}

\begin{notation-and-remark}    \label{def:46}
Let $k$ be a positive integer.

\vspace{3pt}

\noindent
$1^o$ It is customary that for 
$\pi , \rho \in \cP (k)$ one denotes
\[
\pi \wedge \rho := \{ V \cap W \mid V \in \pi, 
\, W \in \rho, \, V \cap W \neq \emptyset \}.
\]
It is immediate that $\pi \wedge \rho$ (called
the {\em meet} of $\pi$ and $\rho$) belongs to 
$\cP (k)$ and is the largest common lower bound
of $\pi$ and $\rho$ with respect to the 
partial order by reverse refinement.

\vspace{3pt}

\noindent
$2^o$ For every subset $S \subseteq \{ 1, \ldots , k \}$ we will
denote by $\pi_{ { }_S }$ the partition in $\cP(k)$ which 
has a block equal to $\{ 1, \ldots , k \} \setminus S$ 
and has a singleton block at $h$ for every $h \in S$.

\vspace{3pt}

\noindent
$3^o$ We will be interested in meets of the form 
$\pi \wedge \pi_{ { }_S }$ for $\pi \in \cP (k)$ and 
$S \subseteq \{ 1, \ldots , k \}$.  We note that the blocks
of $\pi \wedge \pi_{ { }_S }$ are described as follows:

-- every $h \in S$ becomes a singleton block in
$\pi \wedge \pi_{ { }_S }$;

-- for every $V \in \pi$ which is not contained in $S$, 
we have that $V \setminus S$ is a block of
$\pi \wedge \pi_{ { }_S }$.
\end{notation-and-remark}

\vspace{6pt}

\begin{proposition}   \label{prop:47}
$1^o$ The sequence $(U_n)_{n=1}^{\infty}$ is 
exchangeable in $( \cM , \tr )$.  

\vspace{3pt}

\noindent
$2^o$ The sequence $( \diamondU_n )_{n=1}^{\infty}$ is 
exchangeable in $( \cM , \tr )$.  

\vspace{3pt}

\noindent
$3^o$ Let $\bu$ be the function on partitions associated 
to the $U_n$'s and let $\bt$ be the function on partitions 
associated to the $\diamondU_n$'s.  Then for every 
$k \in \bN$ and $\pi \in \cP (k)$ we have
\begin{equation}   \label{eqn:47a}
\bt ( \pi ) = \sum_{S \subseteq \{ 1, \ldots , k \}}
(-1)^{|S|} \bu ( \pi \wedge \pi_{ { }_S } ).
\end{equation}
\end{proposition}

\begin{proof} $1^o$ Let $\ui, \uj : \{ 1, \ldots , k \} \to \bN$
be tuples such that $\Ker ( \ui ) = \Ker (\uj )$.  It is easily seen 
that there exists $\sigma \in S_{\infty}$ such that 
$\gamma_{\uj (h)} = \sigma^{-1} \gamma_{\ui (h)} \sigma$
for every $1 \leq h \leq k$.  Upon applying the representation 
$U$ we therefore find that 
$U_{\uj (1)} \cdots U_{\uj (k)} = U( \sigma )^{-1} \cdot 
\bigl( U_{\ui (1)} \cdots U_{\ui (k)} \bigr) \cdot U ( \sigma )$,
which in turn implies that 
$\tr \bigl( U_{\uj (1)} \cdots U_{\uj (k)} \bigr) = 
\tr \bigl( U_{\ui (1)} \cdots U_{\ui (k)} \bigr)$.

\vspace{3pt}

\noindent
$2^o, \, 3^o$ We will prove that for every $k \geq 1$
and $\ui : \{ 1, \ldots , k \} \to \bN$ we have 
\begin{equation}   \label{eqn:47b}
\tr \bigl( \diamondU_{\ui (1)} \cdots \diamondU_{\ui (k)} \bigr)
= \sum_{S \subseteq \{ 1, \ldots , k \}}
(-1)^{|S|} \, \bu ( \pi \wedge \pi_{ { }_S } ),
\ \mbox{ where $\pi = \Ker ( \ui ) \in \cP (k)$.}
\end{equation}
Since the right-hand side of (\ref{eqn:47b}) only depends on
$\Ker ( \ui )$, this equation will imply that 
$( \diamondU_n )_{n=1}^{\infty}$ is exchangeable, and will also 
provide a formula connecting the functions $\bt$ and $\bu$, which
is precisely what is stated in (\ref{eqn:47a}).

Let us then fix a tuple $\ui : \{ 1, \ldots , k \} \to \bN$,
for which we verify that (\ref{eqn:47b}) holds. We have
\[
\tr \bigl( \diamondU_{\ui (1)} \cdots \diamondU_{\ui (k)} \bigr) = 
\tr \bigl( \, (U_{\ui (1)} - A_0) \cdots (U_{\ui (k)} - A_0) \, \bigr)
\]
\begin{equation}   \label{eqn:47c}
= \sum_{S \subseteq \{ 1, \ldots , k \}} (-1)^{|S|} \,
\tr \bigl( T_1^{(S)} \cdots T_k^{(S)} \bigr),
\end{equation}
where for every $S \subseteq \{ 1, \ldots , k \}$ and $1 \leq h \leq k$
we put
\[
T_h^{(S)} = \left\{   \begin{array}{ll}
A_0, & \mbox{ if $h \in S$,}   \\
U_{\ui (h)}, & \mbox{ if $h \in \{1 , \ldots , k \} \setminus S$.}
\end{array}  \right.
\]

Now, for every $S \subseteq \{ 1, \ldots , k \}$, the trace
$\tr \bigl( T_1^{(S)} \cdots T_k^{(S)} \bigr)$ can be evaluated 
by using Proposition \ref{prop:35}.  It is actually an easy 
verification, left to the reader, that the recipe from 
Proposition  \ref{prop:35} expresses the desired trace in terms 
of the function $\bu$, in the form
\begin{equation}   \label{eqn:47d}
\tr \bigl( T_1^{(S)} \cdots T_k^{(S)} \bigr) = 
\bu ( \pi \wedge \pi_{ { }_S } ), 
\ \ \mbox{ where $\pi := \Ker ( \ui )$.}
\end{equation}
The required formula (\ref{eqn:47b}) follows when we substitute 
(\ref{eqn:47d}) into (\ref{eqn:47c}).
\end{proof}

\vspace{6pt}

\begin{proposition}  \label{prop:48}
The sequence $( \diamondU_n )_{n=1}^{\infty}$ has the 
singleton vanishing property.
\end{proposition}

\begin{proof}  Fix a partition $\pi \in \cP (k)$ which is known
to have a singleton block at $p$, for some $1 \leq p \leq k$.  
Our goal for the proof is to verify that $\bt ( \pi ) = 0$, where, 
same as in the preceding proposition, $\bt$ denotes the function
on partitions associated to $( \diamondU_n )_{n=1}^{\infty}$.

Consider the self-map $\Phi$ of the power set of 
$\{ 1, \ldots , k \}$ which is defined by putting
\begin{equation}   \label{eqn:48a}
\Phi (S) = \left\{ \begin{array}{ll}
S \setminus \{ p \}, &  \mbox{ if $p \in S$,}  \\
S \cup \{ p \}, &  \mbox{ if $p \not\in S$}
\end{array}   \right\}, \ \ \mbox{ for }
S \subseteq \{ 1, \ldots , k \}.
\end{equation}
It is clear that $\Phi$ is a bijection from 
$2^{\{ 1, \ldots , k \}}$ onto itself.  
Another (elementary, but important) detail concerning $\Phi$ is that,
due to our hypothesis that $\{ p \}$ is a block of $\pi$, one has:
\begin{equation}    \label{eqn:48b}
\pi \wedge \pi_{ { }_S } = \pi \wedge \pi_{ { }_{\Phi (S)} },
\ \ \forall \, S \subseteq \{ 1, \ldots , k \}.
\end{equation}

We can then start from the formula for $\bt (\pi )$ provided 
by Proposition \ref{prop:47}, and write:
\begin{align*}
\bt ( \pi ) 
& = \sum_{S \subseteq \{ 1, \ldots , k \}}
(-1)^{|S|} \bu ( \pi \wedge \pi_{ { }_S } )   \\
& = \frac{1}{2} \, \Bigl( 
\sum_{S \subseteq \{ 1, \ldots , k \}}
(-1)^{|S|} \bu ( \pi \wedge \pi_{ { }_S } )
+ \sum_{S \subseteq \{ 1, \ldots , k \}}
(-1)^{|\Phi(S)|} \bu ( \pi \wedge \pi_{ { }_{\Phi (S)} }  ) \Bigr)
\end{align*}

\begin{equation}   \label{eqn:48c}
= \frac{1}{2} \, \Bigl( 
\sum_{S \subseteq \{ 1, \ldots , k \}}
\bigl( (-1)^{|S|} + (-1)^{| \Phi (S) |} \bigr) \, 
\bu ( \pi \wedge \pi_{ { }_S } ) \Bigr),
\end{equation}
where at the second equality sign we used the fact that 
$\Phi$ is bijective, and at the third equality sign 
we invoked (\ref{eqn:48b}).

Finally, it is clear that for every $S \subseteq \{ 1, \ldots , k \}$
we have $| \Phi (S) | = |S| \pm 1$, and hence
$(-1)^{|S|} + (-1)^{| \Phi (S) |} = 0$.  This shows that all the terms
of the sum indicated in (\ref{eqn:48c}) are equal to $0$, and leads to 
the required conclusion that $\bt ( \pi ) = 0$.
\end{proof}

\vspace{6pt}

Since the sequence $( \diamondU_n )_{n=1}^{\infty}$ was now 
found to satisfy the hypotheses of Theorem \ref{thm:44}, we 
can therefore invoke this theorem and draw the following
conclusion.

\vspace{6pt}

\begin{corollary}   \label{cor:49}
{\em (CLT  for the sequence of elements $\diamondU_n$.)}
Consider the elements
\begin{equation}   \label{eqn:49a}
s_n := \frac{1}{\sqrt{n}} \, \sum_{i=1}^n \diamondU_i
\ = \ \frac{1}{\sqrt{n}} \, \sum_{i=1}^n 
(U ( \gamma_i) - A_0) \in \cM,
\ \ n \geq 1.
\end{equation}
On the other hand, let $\bt$ be the function on partitions 
associated to $( \diamondU_n )_{n=1}^{\infty}$ and consider the
linear functional $\mu : \bC [X] \to \bC$ defined via the 
prescription that $\mu (1) = 1$ and that 
\begin{equation}   \label{eqn:49b}
\mu ( X^k ) =
\sum_{ \rho \in \cP_2 (k)} \ \bt ( \rho ),
\ \ \forall \, k \in \bN ,
\end{equation}
with the empty sum on the right-hand side of 
(\ref{eqn:49b}) being read as $0$ for $k$ odd.  
Then $\mu$ is a positive functional and the $s_n$'s 
from (\ref{eqn:49a}) converge to $\mu$ in moments.
\hfill $\square$
\end{corollary}

\vspace{6pt}

\begin{remark}    \label{rem:410}
Quite obviously, the law of the element $s_n$ indicated in 
Equation (\ref{eqn:49a}) is the same as the law $\mu_n$ mentioned
in Theorem \ref{thm:25}.  We would therefore like to claim that
we have now obtained the proof of Theorem \ref{thm:25}, where 
the linear functional $\mu$ indicated in the preceding corollary 
is an algebraic incarnation of the law ``$\muw$'' from 
Theorem \ref{thm:25}.  There is however a difference between the 
$\mathrm{weak}^{*}$-convergence needed for the conclusion of 
Theorem \ref{thm:25}, and the convergence in moments which is provided 
by Corollary \ref{cor:49}.  The well-known argument to be invoked 
in such a situation (see e.g.~\cite[Theorem 30.2]{Bill1995}) is that 
the limit law $\mu$ is uniquely determined by its moments.  This in 
turn is known to follow (cf.~\cite[Theorem 30.1]{Bill1995}) if we 
can provide some bounds on the moments of $\mu$ which ensure that the
series $\sum_{k=0}^{\infty} ( \mu (X^k) / k! ) z^k$ has a positive 
radius of convergence.

In order to get the desired bounds on the moments $\mu (X^k)$, 
let us look back at the functions on partitions $\bu$ and $\bt$ that
were introduced in Proposition \ref{prop:47}.3.  We first observe
that $| \bu ( \pi ) | \leq 1$ for every 
$\pi \in \sqcup_{k=1}^{\infty} \cP (k)$, as is clear from the fact that 
$\bu ( \pi )$ is computed as the value of $\chi$ on a suitable product
of star transpositions.  By plugging this observation into the formula 
(\ref{eqn:47a}) of Proposition \ref{prop:47}.3 we therefore find that
\[
| \bt ( \pi ) | \leq \sum_{S \subseteq \{ 1, \ldots , k \}}
| \bu ( \pi \wedge \pi_{{ }_S} ) | \leq 2^k, \ \ \forall
\, k \in \bN \mbox{ and } \pi \in \cP (k).
\]
The latter inequality can in turn be plugged into Equation 
(\ref{eqn:49b}) of Corollary \ref{cor:49} to conclude that 
for every $k \in \bN$ we have
\begin{equation}   \label{eqn:410a}
| \mu (X^k) | \leq 2^k \cdot | \cP (k) | =
\left\{   \begin{array}{ll}
0, & \mbox{ if $k$ is odd,}  \\
2^k \cdot (k-1)!!, & \mbox{ if $k$ is even.}
\end{array}   \right.
\end{equation}

The bound obtained in (\ref{eqn:410a}) is clearly sufficient 
for our purposes.  Indeed, for $k$ even we have
$\bigl( 2^k \cdot (k-1)!! \bigr) / k! = 2^{k/2} / (k/2)!$,
which shows that the series 
$\sum_{k=0}^{\infty} ( \mu (X^k) / k! ) z^k$ has an
infinite radius of convergence.  Thus Corollary \ref{cor:49},
together with this bound on moments, implies Theorem \ref{thm:25}.
\end{remark}

\vspace{6pt}

\begin{remark}   \label{rem:411}
We mention that the idea of associating a function on 
pair-partitions to an extremal character of $S_{\infty}$
has been considered a while ago in \cite{BoGu2002}.  See 
pages 215-216 (in particular the figure on page 216) 
of \cite{BoGu2002}, where it is explained how that function 
on pair-partitions is computed.
Calculations made on various special examples don't seem 
to suggest, however, a connection between the construction 
from \cite{BoGu2002} and the functions on partitions $\bt$ 
and $\bu$ that appeared in Proposition \ref{prop:47} of the 
present paper.
\end{remark}

$\ $

\section{Moments of the limit law
$\muw$, in terms of permutations $\tau_{\pi}$}

\noindent
In the present section we do some further processing
of the formula (\ref{eqn:49b}) for an even moment of the limit law 
$\muw$.  Our goal, reached in Proposition \ref{prop:56}, is to
describe such a moment in a way which only refers to how the 
character $\vvarphi : S_{\infty} \to \bC$ evaluates products of 
star-transpositions ``$\tau_{\pi}$'' of the kind discussed in 
Section 1.4 of the Introduction.  It turns out to be convenient 
to consider such permutations $\tau_{\pi}$ not only when $\pi$ is
a pair-partition (as mentioned in Section 1.4) but also in the case
when $\pi$ has some singleton blocks.

\vspace{6pt}

\begin{notation}   \label{def:51}
Let $k$ be a positive integer.  We denote by 
$\cP_{\leq 2} (k)$ the set of all
partitions $\pi \in \cP (k)$ such that every
block $V \in \pi$ has cardinality $1$ or $2$ -- 
that is, $V$ is either a ``singleton-block'' or a ``pair''.  
For $\pi \in \cP_{\leq 2} (k)$, the number of 
singleton-blocks in $\pi$ will be denoted as $| \pi |_1$,
and the number of pairs in $\pi$ will be denoted as
$| \pi |_2$.  We thus have 
\[
| \pi |_1 + | \pi |_2 = | \pi |,
\mbox{ the total number of blocks of $\pi$, while}
\]
\[
| \pi |_1 + 2 \, | \pi |_2 = k,
\mbox{ the total number of points partitioned by $\pi$}.
\]
Note that, as a consequence of the latter equation, 
$| \pi |_1$ is sure to be of same parity as $k$. 
\end{notation}

\vspace{6pt}

\begin{definition}    \label{def:52}
{\em (The permutation $\tau_{\pi}$.)}
Let $k \in \bN$ and let $\pi \in \cP_{\leq 2} (k)$.
We construct a permutation $\tau_{\pi} \in S_{\infty}$,
defined as a product of star-transpositions $\gamma_m$, 
as follows.

\vspace{3pt}

{\em Step 1.} Denoting $| \pi | =: \ell$, we list the blocks
$V_1, \ldots , V_{\ell}$ of $\pi$ in decreasing order of 
their maximal elements, and for $1 \leq i \leq \ell$ we put 
\begin{equation}   \label{eqn:52a}
a_i := \min (V_i), \ \ b_i := \max (V_i).
\end{equation}
The convention on the order in which we consider
$V_i$'s reflects on the $b_i$'s, and gives:
\begin{equation}   \label{eqn:52b}
k = b_1 > b_2 > \cdots > b_{\ell}.
\end{equation}

\vspace{3pt}

{\em Step 2.} Let 
$\ur : \{ 1, \ldots , k \} \to \{ 1, \ldots , \ell \}$ be the  
tuple which takes the value $i$ on the block $V_i$, for every 
$1 \leq i \leq \ell$.  In other words, $\ur$ is described by the 
formula 
\begin{equation}   \label{eqn:52c}
\ur( a_i ) = \ur( b_i ) = i, \ \ \forall \, 1 \leq i \leq \ell.
\end{equation}

\vspace{3pt}

{\em Step 3.}  We put
\begin{equation}    \label{eqn:52d}
\tau_{\pi} := \gamma_{\ur(1)} \cdots \gamma_{\ur(k)} 
\in S_{\ell + 1} \subseteq S_{\infty}. 
\end{equation}
The relation $\tau_{\pi} \in S_{\ell + 1}$ holds because
(\ref{eqn:52d}) only uses
$\gamma_1, \gamma_2, \ldots , \gamma_{\ell}$, 
with each of them appearing either once or two times in the
product indicated there; as a consequence, $\tau_{\pi}$ cannot
move any number $n > \ell + 1$. 
\end{definition}

\vspace{6pt}

\begin{example}   \label{example:53}
For a concrete example, suppose that $k=7$ and that 
we are dealing with
\[
\pi = \bigl\{ \{ 1,6 \}, \{ 2,5 \}, \{ 3 \}, \{ 4,7 \} \bigr\}
\in \cP_{\leq 2} (7).
\]
In Step 1 of the preceding definition we thus have $\ell = 4$,
and the order in which we consider the blocks of $\pi$ is
\[
V_1 = \{ 4,7 \}, \, V_2 = \{ 1,6 \}, 
\, V_3 = \{ 2,5 \}, \, V_4 = \{ 3 \}.
\]
For a pictorial rendering of Steps 2 and 3, Figure 2 shows a 
drawing of $\pi$ where $\gamma_1$'s are placed on top of $V_1$, 
$\gamma_2$'s are placed on top of $V_2$, and so on.  The 
resulting product of star-transpositions is, for this example:
$\tau_{\pi} = 
\gamma_2 \gamma_3 \gamma_4 \gamma_1 \gamma_3 \gamma_2 \gamma_1 
= (1,5,4,2) \in S_5$.

\begin{center}
\setlength{\unitlength}{0.7cm}
$\pi = \
\begin{picture}(8,2) \thicklines
   \put(0,-1){\line(0,1){2}}
   \put(0,-1){\line(1,0){5}}
   \put(1,-0.5){\line(0,1){1,5}}
   \put(1,-0.5){\line(1,0){3}}
   \put(2,0){\line(0,1){1}}
   \put(3,0){\line(0,1){1}}
   \put(3,0){\line(1,0){3}}
   \put(4,-0.5){\line(0,1){1.5}}
   \put(5,-1){\line(0,1){2}}
   \put(6,0){\line(0,1){1}}
   \put(-0.3,1.2){$\gamma_2$}
   \put(0.7,1.2){$\gamma_3$}
   \put(1.7,1.2){$\gamma_4$}
   \put(2.7,1.2){$\gamma_1$}
   \put(3.7,1.2){$\gamma_3$}
   \put(4.7,1.2){$\gamma_2$}
   \put(5.7,1.2){$\gamma_1$}
\end{picture}$

\vspace{1cm}

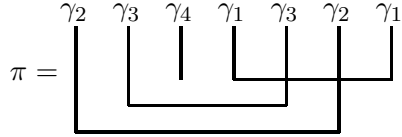
{\bf Figure 2.} {\em A picture of the 
$\pi \in \cP_{\leq 2} (7)$ used in 
Example \ref{example:53}.}
\end{center}
\end{example}

\vspace{6pt}

\begin{remark}    \label{rem:54}
Let $\pi$ be a partition in $\cP_{\leq 2} (k)$. The 
specifics of how we constructed the permutation
$\tau_{\pi}$ will be useful later on in the paper; but 
if we only want to know how the character $\vvarphi$
applies to $\tau_{\pi}$, then let us note that one 
simply has
\begin{equation}   \label{eqn:54a}
\vvarphi ( \tau_{\pi} ) = \bu ( \pi ), 
\end{equation}
where $\bu$ is the function on partitions associated to the
sequence $( U_n )_{n=1}^{\infty}$ (as considered in 
Proposition \ref{prop:47}).  In other words, one can write
\begin{equation}   \label{eqn:54b}
\vvarphi ( \tau_{\pi} )
= \tr ( U_{\ui (1)} \cdots U_{\ui (k)} )
= \vvarphi ( \gamma_{\ui(1)} \cdots \gamma_{\ui(k)} )
\end{equation}
where $\ui : \{ 1, \ldots , k \} \to \bN$ is any tuple
such that $\Ker ( \ui ) = \pi$.  This is because the permutation 
$\gamma_{\ui(1)} \cdots \gamma_{\ui(k)}$ appearing in 
(\ref{eqn:54b}) belongs to the same conjugacy class of 
$S_{\infty}$ as $\tau_{\pi}$, and thus the character $\vvarphi$
must take the same value on them.
\end{remark}

\vspace{6pt}

Now let us fix an even $k \in \bN$ and let us look again at the 
formula $(\ref{eqn:49b})$ obtained in Section 3 for the 
moment of order $k$ of our limit law of interest.  In order to 
move further, let us focus on one of the terms $\bt ( \rho )$ 
appearing on the right-hand side of (\ref{eqn:49b}).  We have at
our disposal a formula, established in Proposition \ref{prop:47}.3,
which expresses $\bt ( \rho )$ as a summation over subsets 
$S \subseteq \{ 1, \ldots , k \}$.  Our next point is that the 
latter summation actually has a lot of cancellations, and can 
be simplified in the way indicated by the next lemma.

\vspace{6pt}

\begin{lemma}  \label{lemma:55}
Let $\bt$ be the function on partitions associated to the sequence
$( \diamondU_n )_{n=1}^{\infty}$ (same as in Section 4.2).  Let $k$ 
be an even positive integer, and let $\rho$ be in $\cP_2 (k)$.  Then
\begin{equation}   \label{eqn:55a}
\bt ( \rho ) 
= \sum_{\substack{\pi \in \cP_{\leq 2}(k)  \\ 
  \pi \leq \rho}} (-1)^{| \pi |_1 /2} \vvarphi ( \tau_{\pi} ),
\end{equation}
where the non-negative even integer $| \pi |_1$ is picked from 
Notation \ref{def:51}.
\end{lemma}

\begin{proof}
We start from the formula for $\bt ( \rho )$ provided by
Equation (\ref{eqn:47a}) of Proposition \ref{prop:47}, and in that
formula we group terms according to what is $\rho \wedge \pi_{ { }_S }$.
Noting that $\rho \wedge \pi_{ { }_S } \leq \rho$ for all 
$S \subseteq \{ 1, \ldots , k \}$ (which implies in particular that
$\rho \wedge \pi_{ { }_S } \in \cP_{\leq 2} (k)$), we find that
\[
\bt ( \rho ) =
\sum_{ \substack{ \pi \in \cP_{\leq 2} (k),  \\
                  \pi \leq \rho }                  } 
\ \Bigl( \, \sum_{ 
  \substack{ S \subseteq \{ 1, \ldots , k \} \ \mathrm{with}  \\
              \rho \wedge \pi_{ { }_S } = \pi }      } 
\ (-1)^{|S|} \, \Bigr)  \bu ( \pi ).
\]
The observation made in (\ref{eqn:54a}) of 
Remark \ref{rem:54} allows us to re-write this as 
\begin{equation}   \label{eqn:55b}
\bt ( \rho ) =
\sum_{ \substack{ \pi \in \cP_{\leq 2} (k),  \\
                  \pi \leq \rho }                  } 
\ \Bigl( \, \sum_{ 
  \substack{ S \subseteq \{ 1, \ldots , k \} \ \mathrm{with}  \\
              \rho \wedge \pi_{ { }_S } = \pi }      } 
\ (-1)^{|S|} \, \Bigr)  \vvarphi ( \tau_{\pi} ).
\end{equation}
The cancellations we need are
in the inside sum on the right-hand side of (\ref{eqn:55b}).
More precisely, we will prove that for every 
$\pi \in \cP_{\leq 2} (k)$ such that $\pi \leq \rho$ one has:
\begin{equation}    \label{eqn:55c}
\sum_{ \substack{S \subseteq \{ 1, \ldots ,k \} \ \mathrm{with}  \\
                 \rho \wedge \pi_{ { }_S} = \pi }        } 
\ (-1)^{|S|} \ = \ (-1)^{| \pi |_1 /2}.
\end{equation}
It is clear that plugging (\ref{eqn:55c}) into (\ref{eqn:55b})
will lead to the formula (\ref{eqn:55a}) claimed by the lemma.

\vspace{6pt}

We are thus left to fix a partition $\pi \in \cP_{\leq 2} (k)$ such that 
$\pi \leq \rho$, for which we verify that (\ref{eqn:55c}) holds.  Let us
write explicitly $\rho = \{ V_1, \ldots , V_p, W_1, \ldots , W_q \}$, 
where the pairs $V_1, \ldots , V_p$ are also appearing in $\pi$, while 
each of $W_1, \ldots , W_q$ is broken into two singleton-blocks of $\pi$ (which,
in particular, forces the relation $| \pi |_1 = 2q$).
We leave it as an exercise to the reader to check the elementary fact that
for a subset $S \subseteq \{ 1, \ldots , k \}$ one has: 
\begin{equation}   \label{eqn:55d}
\Bigl( \rho \wedge \pi_{ { }_S } = \pi \Bigr)
\ \Leftrightarrow \ \Bigl( \begin{array}{c}
S \cap V_i = \emptyset \mbox{ for every $1 \leq i \leq p$, and}  \\
S \cap W_j \neq \emptyset \mbox{ for every $1 \leq j \leq q$ }
\end{array}   \Bigr) .
\end{equation}
By following the description indicated on the right-hand side of 
(\ref{eqn:55d}), it is immediate that the sets 
$S \subseteq \{ 1, \ldots , k \}$ such that $\rho \wedge \pi_{ { }_S } = \pi$
are enumerated, without repetitions, by taking the following steps:
\begin{equation}   \label{eqn:55e}
\begin{array}{lcl}
\vline & \mbox{Step 1.} 
       & \mbox{Pick an integer $q'$ such that $0 \leq q' \leq q$.}          \\
\vline & \mbox{Step 2.}
       &  \mbox{Select $q'$ of the pairs $W_1, \ldots , W_q$.  Say we selected the} \\
\vline &        
       &  \mbox{ $\ $ pairs $\{ W_j \mid j \in Q' \}$,
            where $Q' \subseteq \{ 1, \ldots , q \}$ has $|Q'| = q'$.}    \\
\vline & \mbox{Step 3.}
       & \mbox{Choose one element of each of the pairs selected in Step 2.}   \\
\vline &
       & \mbox{ $\ $ That is, for every $j \in Q'$ choose a $c_j \in W_j$.}        \\
\vline & \mbox{Step 4.} 
       & \mbox{Put $S := \{ c_j \mid j \in Q' \} \cup 
                     \Bigl( \cup_{j \in Q''} W_j \Bigr)$, where 
                     $Q'' = \{ 1, \ldots , q \} \setminus Q'$.}
\end{array}
\end{equation}
Note that every $S$ produced by (\ref{eqn:55e}) has a 
cardinality $|S|$ of the same parity as $q'$.

So then, the sum on the left-hand side of (\ref{eqn:55c}) is 
found to be equal to:
\begin{equation}   \label{eqn:55f}
\sum_{q' = 0}^q \binom{q}{q'}
\cdot 2^{q'} \cdot (-1)^{q'},
\end{equation}
where: 
the binomial coefficient corresponds to the choice of $Q'$ in Step 2; 
the factor $2^{q'}$ corresponds to the choices of $c_j$'s made in Step 3;
and the factor $(-1)^{q'}$ corresponds to the $(-1)^{|S|}$ from the
summation on the left-hand side of (\ref{eqn:55c}).

But (\ref{eqn:55f}) is, however,
the binomial expansion for $\bigl( 1 + 2 \cdot (-1) \bigr)^q$, 
and is thus equal to $(-1)^q$.  Since 
$q = | \pi |_1 /2$, we have indeed
arrived to the right-hand side of (\ref{eqn:55c}).
\end{proof}

\vspace{6pt}

The new form of Equation (\ref{eqn:49b}) that was announced 
at the beginning of the section is then stated as follows.

\vspace{6pt}

\begin{proposition}   \label{prop:56}
For an even $k \in \bN$, the limit law $\muw$ from Theorem 
\ref{thm:25} has
\begin{equation}    \label{eqn:56a}
\int_{\bR} t^k \, d \muw (t) 
= \sum_{\pi \in \cP_{\leq 2} (k)} 
(-1)^{ | \pi |_1/2 } \cdot ( | \pi |_1 -1 )!! \cdot \vvarphi ( \tau_{\pi} ),
\end{equation}
where the non-negative even integer $| \pi |_1$ is picked from
Notation \ref{def:51}.
\end{proposition}

\begin{proof}  As discussed in Remark \ref{rem:410}, the moment of order
$k$ of $\muw$ can be viewed as $\mu (X^k)$, where $\mu : \bC [X] \to \bC$ 
is the linear functional appearing in Corollary \ref{cor:49}.  Moreover,
Equation (\ref{eqn:49b}) of Corollary \ref{cor:49} expresses $\mu (X^k)$
as a summation over $\rho \in \cP_2 (k)$.  In that summation we replace
every term by using Lemma \ref{lemma:55}, and then we interchange the order
of the summations over $\rho$ and $\pi$.  We get
\[
\sum_{\pi \in \cP_{\leq 2} (k)}
\Bigl( \, \sum_{\substack{\rho \in \cP_2(k) 
                   \\ \rho \geq \pi}} (-1)^{| \pi |_1 /2} 
\vvarphi ( \tau_{\pi} ) \, \Bigr).
\]
This is exactly the quantity indicated on the right-hand side of
(\ref{eqn:56a}), since for every $\pi \in \cP_{\leq 2} (k)$
there are $( | \pi |_1 -1 )!!$ pairings 
$\rho \in \cP_2 (k)$ such that $\rho \geq \pi$.
\end{proof}

$\ $

\section{Combinatorics: $\sigma_{\pi}$,
$\tau_{\pi}$ and a relation between them}

\begin{notation-and-remark}    \label{def:61}
{\em (The permutation $\sigma_{\pi}$.)}
Consider a partition $\pi \in \cP_{\leq 2} (k)$, which we write
explicitly in the same way as in Step 1 of Definition \ref{def:52}:
$\pi = \{ V_1, \ldots , V_{\ell} \}$, with $\min (V_i) =: a_i$
and $\max (V_i) =: b_i$ for all $1 \leq i \leq \ell$,
and where $k = b_1 > \cdots > b_{\ell}$.  In Definition \ref{def:52}
we introduced a permutation $\tau_{\pi}$ associated to $\pi$, and this
was important for discussing the moments of the limit law $\muw$ of 
our Theorem \ref{thm:25}.  But frequently encountered in the 
literature there is however another permutation associated to $\pi$, 
which we will denote as $\sigma_{\pi}$, and is simply defined as 
\begin{equation}    \label{eqn:61a}
\sigma_{\pi} := 
\prod_{ \substack{  1 \leq i \leq \ell \\
                      \mathrm{with} \ a_i \neq b_i }  } 
\ (a_i, b_i) \in S_k \subseteq S_{\infty}.
\end{equation}

At first sight,the permutations $\sigma_{\pi}$ and $\tau_{\pi}$ 
do not appear to be related to each other.  But it turns out that 
there exists a non-trivial connection, at the level of orbit structures, 
between $\tau_{\pi}$ and the product $\cyc_{k+1} \cdot \sigma_{\pi}$, 
where $\cyc_{k+1} \in S_{k+1}$ is the forward cycle from 
Notation \ref{def:21}.  We will first explain how this goes on 
a concrete example (cf.~Example \ref{example:63}), then the general 
case is addressed in Proposition \ref{prop:66} at the end of the
section.
\end{notation-and-remark}

\vspace{6pt}

\begin{notation}    \label{def:62}
Notation as above.  In addition to $a_1, \ldots , a_{\ell}$
and $b_1, \ldots , b_{\ell}$ we put
\begin{equation}   \label{eqn:62a}
a_0 = b_0 := k+1,
\end{equation}
and we denote 
\begin{equation}   \label{eqn:62b}
B_{\pi} := \{ b_{\ell}, \ldots, b_1, b_0 \} 
\subseteq \{ 1, \ldots , k+1 \}.
\end{equation}
\end{notation}

\vspace{6pt}

\begin{example}   \label{example:63}
Let us look again at Example \ref{example:53}
of the preceding section, where we had $k=7$ and
$\pi = \bigl\{ \{ 1,6 \}, \{ 2,5 \}, \{ 3 \}, \{ 4,7 \} \bigr\}
\in \cP_{\leq 2} (7)$, with blocks listed in the order
\[
V_1 = \{ 4,7 \}, \, V_2 = \{ 1,6 \}, 
\, V_3 = \{ 2,5 \}, \, V_4 = \{ 3 \}.
\]
In this example, the set $B_{\pi}$ from (\ref{eqn:62b}) comes out as
$B_{\pi} = \{ 3,5,6,7,8 \}$.

We have $\sigma_{\pi} = (1,6)(2,5)(4,7) \in S_7$.  In order
to connect to $\tau_{\pi}$, it turns out to be relevant to 
examine how the set $B_{\pi}$ is partitioned by its 
intersections with orbits of $\cyc_{k+1} \cdot \sigma_{\pi}$. 
In the example at hand we have
\[
\cyc_8 \cdot \sigma_{\pi} = (1,2, \ldots ,8) \cdot (1,6)(2,5)(4,7) 
= (1,7,5,3,4,8) (2,6) \in S_8,
\]
and intersecting $B_{\pi}$ with the orbits of
$\cyc_8 \cdot \sigma_{\pi}$ thus gives
\[
B_{\pi} = \{ 3,5,7,8 \} \cup \{ 6 \},
\]
a partition of type $4+1$.  Now, we saw in 
Example \ref{example:53} that $\tau_{\pi} = (1,5,4,2) (3) \in S_5$,
which also has an orbit structure of type $4+1$.  
Proposition \ref{prop:66} below claims this is not a coincidence.
In order to arrive to that general statement, we will use a lemma,
preceded by another bit of notation.
\end{example}

\vspace{6pt}

\begin{notation}[{\em Induced permutation}]   \label{def:64}
Let $\tau \in S_{\infty}$ and let $A \subseteq \bN$ be finite and 
non-empty.  We will use the notation $\left. \tau \right|_A$ for 
the permutation in $S_{\infty}$ which is obtained by only retaining
how $\tau$ acts on $A$.

If $A$ happens to be invariant with respect to $\tau$, then 
$\left. \tau \right|_A$ is simply obtained by keeping the cycles
of $\tau$ that are contained in $A$, while every 
$n \in \bN \setminus A$ becomes a fixed point.

In general, $\left. \tau \right|_A$ maps every $a \in A$ to the 
first element of the orbit $\tau (a), \tau^2 (a), \ldots$ 
which is found to be again in $A$ (while every $n \in \bN \setminus A$
becomes a fixed point).  This is very nicely followed in cycle notation:
in order to obtain $\left. \tau \right|_A$ we start from the cycle
notation for $\tau$, and remove the elements of $\bN \setminus A$ 
that might appear in it.

\noindent
[A concrete example: if $\tau = (1,3,2)(5,6)$ and $A = \{ 1,2,5,7 \}$, 
then $\left. \tau \right|_A = (1, \not3, 2)(5, \not6) = (1,2)(5) = (1,2)$.]
\end{notation}

$\ $

\begin{lemma}   \label{lemma:65}
Let $\pi \in \cP_{\leq 2} (k)$ and consider the
permutations $\sigma_{\pi}$, $\tau_{\pi}$ introduced above
(in Notation \ref{def:61} and in Definition \ref{def:52},
respectively).  Let us form the induced permutation 
\begin{equation}   \label{eqn:65a}
\left. (\sigma_{\pi} \cdot \cyc_{k+1}^{-1} ) 
\right|_{B_{\pi}} \in S_{k+1},
\end{equation}
where $\cyc_{k+1} \in S_{k+1}$ is the forward cycle and
$B_{\pi} = \{ b_{\ell}, \ldots , b_1, b_0 \} \subseteq 
\{ 1, \ldots , k+1 \}$ is as in Notation \ref{def:62}. 
The permutation (\ref{eqn:65a}) relates to $\tau_{\pi}$ 
in the following way: 
\begin{equation}    \label{eqn:65b}
\left\{   \begin{array}{l}
\mbox{if $i,j \in \{0, \ldots, \ell\}$
are such that } 
\bigl( \left. ( \sigma_{\pi} \cdot \cyc_{k+1}^{-1} ) 
\right|_{B_{\pi}} \bigr) \, (b_i) = b_j,              \\ 
\mbox{ then it follows that $\tau_{\pi} (i+1) = j+1$.}
\end{array}  \right.
\end{equation}
\end{lemma}

\begin{proof}  The case when in (\ref{eqn:65b}) we have $i=0$ 
is slightly different from the others.  But we can afford to 
not discuss this case, due to the following observation: if we
assume that all cases with $i \neq 0$ have been verified, then 
the case $i=0$ automatically follows.  Indeed, the permutation
$\left. ( \sigma_{\pi} \cdot \cyc_{k+1}^{-1} ) 
\right|_{B_{\pi}}$ gives in particular a bijection
$f$ from the set $B_{\pi} = \{ b_0, b_1, \ldots , b_{\ell} \}$ 
onto itself.  Denote
$f(b_0) = b_{j_0}, f(b_1) = b_{j_1}, 
\ldots , f(b_{\ell}) = b_{j_{\ell}}$,
(where $\{ j_0, \ldots , j_{\ell} \}$ is a permutation of
$\{ 0, 1, \ldots , \ell \}$) and suppose it was verified that 
\[
\pi_{\tau} (1+1) = j_1 + 1, \,  \pi_{\tau}  (2+1) = j_2 +1, 
\ldots , \pi_{\tau} ( \ell +1 ) = j_{\ell} + 1.
\]
By equating the set-differences
\[
\{ 1, \ldots , \ell + 1 \} \setminus
\{ \tau_{\pi} (2), \ldots , \tau_{\pi} (\ell + 1) \} 
= \{ 1, \ldots , \ell + 1 \} \setminus
\{ j_1 +1, \ldots , j_{\ell} + 1 \}
\]
we then find that $\tau_{\pi} (1) = j_0 + 1$, which is precisely the 
missing case $i=0$ of (\ref{eqn:65b}).

\vspace{6pt}

For the remaining part of the proof we fix 
$i,j \in \{ 0, \ldots , \ell \}$, with $i \neq 0$, and
where $\bigl( \left. ( \sigma_{\pi} \cdot \cyc_{k+1}^{-1} ) 
\right|_{B_{\pi}} \bigr) \, (b_i) = b_j$.  The latter equation
entails the existence of a $q \geq 1$ such that 
\begin{equation}    \label{eqn:65c}
( \sigma_{\pi} \cdot \cyc_{k+1}^{-1} )^q (b_i) = b_j, \mbox{ while } 
( \sigma_{\pi} \cdot \cyc_{k+1}^{-1} )^p  (b_i) \not\in B_{\pi}
\mbox{ for every $1 \leq p < q$.} 
\end{equation}
Based on (\ref{eqn:65c}), we have to prove that 
$\tau_{\pi} (i+1) = j+1$.  

\vspace{6pt}

{\em Subcase $q=1$.}
It is instructive to work out in detail the case when 
in (\ref{eqn:65c}) we have $q=1$, i.e.~we simply have 
$( \sigma_{\pi} \cdot \cyc_{k+1}^{-1} ) (b_i) = b_j$.  By applying 
$\sigma_{\pi}$ to both sides of this equation, and upon recalling 
that $\sigma_{\pi}$ swaps $a_j$ and $b_j$ for every 
$0 \leq j \leq \ell$, we find that
\begin{equation}   \label{eqn:65d}
\cyc_{k+1}^{-1} (b_i) = a_j.
\end{equation}
Let us assume for a moment that in (\ref{eqn:65d}) we have
$b_i > 1$, and the equation thus says that $a_j = b_i - 1$.
On the other hand, let us look at the permutation $\tau_{\pi}$
and let us follow how it acts on the number $i+1$.  Recall 
from Definition \ref{def:52} (cf.~Equation (\ref{eqn:52c})) that
$\tau_{\pi}$ is defined as a product of star-transpositions,
which we write here by putting into evidence the factors on the
consecutive positions $a_j$ and $b_i$:
\begin{equation}    \label{eqn:65e}
\tau_{\pi} = \bigl( \prod_{1 \leq h < a_j} \gamma_{\ur(h)} \bigr) 
\cdot \gamma_{\ur( a_j)} \gamma_{\ur (b_i)} \cdot
\bigl( \prod_{b_i < h \leq k} \gamma_{\ur(h)} \bigr). 
\end{equation}
Concerning the product (\ref{eqn:65e}), we make the following
observations.

(i) The part $\prod_{b_i < h \leq k} \gamma_{\ur(h)}$ does not move
$i+1$ at all, since it has no factors ``$\gamma_i$''.

(ii) The part $\gamma_{\ur( a_j)} \gamma_{\ur (b_i)}$ comes to
$\gamma_j \gamma_i$, and sends $i+1$ to $j+1$:
$\gamma_j \bigl( \gamma_i (i+1) \bigr) = \gamma_j (1) = j + 1$.

\hspace{0.3cm}
So from here on we have to watch for what happens to $j+1$.  But:

(iii) The part $\prod_{1 \leq h < a(j)} \gamma_{\ur(h)}$ of 
(\ref{eqn:65e}) does not move $j+1$, since it has no factors
``$\gamma_j$''.

\noindent
So overall, the conclusion is that $\tau_{\pi} (i+1) = j+1$, 
as we wanted to prove.

\vspace{6pt}

In order to complete the discussion of the case $q=1$, 
we must also examine the subcase when in (\ref{eqn:65d}) we 
have $b_i = 1$.  This happens precisely when the partition $\pi$ 
has a singleton-block at $\{ 1 \}$, which appears last in the 
numbering of blocks of $\pi$: $V_{\ell} = \{ 1 \}$, with 
$a_{\ell} = b_{\ell} = 1$. Thus it must be that $i = \ell$, while 
the equation $a_j = \cyc_{k+1}^{-1} (1) = k+1$ implies $j=0$. Our 
desired conclusion about the action of $\tau_{\pi}$ thus comes to
$\tau_{\pi} ( \ell + 1) = 1$.  This is indeed true, and is verified 
by writing
$\tau_{\pi} = \gamma_{\ur (1)} 
\cdot \bigl( \prod_{2 \leq h \leq k} \gamma_{\ur(h)} \bigr)$,
where
$\prod_{2 \leq h \leq k} \gamma_{\ur(h)}$
does not move $\ell + 1$, while $\gamma_{\ur (1)}$ is $\gamma_{\ell}$,
sending $\ell + 1$ to $1$.

\vspace{6pt}

{\em Subcase $q>1$.}
In what follows we assume that $q$ of (\ref{eqn:65c}) is $\geq 2$.  

It is clear that the orbit of $b_i$ under the action of
$\sigma_{\pi} \cdot \cyc_{k+1}^{-1}$ is contained in $\{ 1, \ldots , k+1 \}$.
Hence the numbers $( \sigma_{\pi} \cdot \cyc_{k+1}^{-1} )^p  (b_i)$ with
$1 \leq p < q$ are in $\{ 1, \ldots , k+1 \} \setminus B_{\pi}$, which 
forces every such $( \sigma_{\pi} \cdot \cyc_{k+1}^{-1} )^p  (b_i)$ to be of
the form $a_{i(p)} = \min ( V_{i(p)} )$ for a pair 
$V_{i(p)} = \{ a_{i(p)}, b_{i(p)} \}$ of $\pi$.  We thus have:
\begin{equation}   \label{eqn:65f}
\left\{   \begin{array}{l}
( \sigma_{\pi} \cdot \cyc_{k+1}^{-1} )  (b_i) = a_{i(1)},
\ ( \sigma_{\pi} \cdot \cyc_{k+1}^{-1} ) ( a_{i(1)} ) = a_{i(2)}, \ldots  \\
                    \\
\mbox{ $\ $} \ldots, 
( \sigma_{\pi} \cdot \cyc_{k+1}^{-1} ) ( a_{i(q-2)} ) = a_{i(q-1)},
\ ( \sigma_{\pi} \cdot \cyc_{k+1}^{-1} ) ( a_{i(q-1)} ) = b_j.
\end{array}  \right.
\end{equation}

Similarly to how we did in (\ref{eqn:65d}) above, we apply $\sigma_{\pi}$
to both sides of each of the equalities listed in (\ref{eqn:65f}), and 
we arrive to: 
\begin{equation}   \label{eqn:65g}
\cyc_{k+1}^{-1}  (b_i) = b_{i(1)},
\ \cyc_{k+1}^{-1} ( a_{i(1)} ) = b_{i(2)}, \ldots ,
\ \cyc_{k+1}^{-1} ( a_{i(q-2)} ) = b_{i(q-1)},
\ \cyc_{k+1}^{-1} ( a_{i(q-1)} ) = a_j.
\end{equation}
With the possible exception of the last one listed, the instances
of $\cyc_{k+1}^{-1}$ that appear in (\ref{eqn:65g}) are genuine subtractions 
by $1$.  This is because all of $b_{i(1)}, \ldots , b_{i(q-1)}$ are 
elements of some blocks of $\pi$, and are thus numbers from 
$\{ 1, \ldots , k \}$, which cannot be equal to $k+1$.  We may 
therefore re-write (\ref{eqn:65e}) in the form
\begin{equation}   \label{eqn:65h}
b_i - 1 = b_{i(1)},
\  a_{i(1)}   - 1 = b_{i(2)}, \ldots ,
\  a_{i(q-2)} - 1 = b_{i(q-1)},
\  \cyc_{k+1}^{-1} ( a_{i(q-1)} ) = a_j,
\end{equation}
where for the last occurrence of $\cyc_{k+1}^{-1}$ we have to also 
consider the possibility that $i(q-1) = \ell$ (with $a_{\ell} = 1$)
and $j=0$ (with $a_0 = k+1$).

We now move to examine how the permutation $\tau_{\pi}$ acts 
on the number $i+1$.  Assume first that we are in the
situation with $a_{i(q-1)} > 1$, when the last equality in 
(\ref{eqn:65h}) just says that $a_j = a_{i(q-1)} -1$.
We break the product defining $\tau_{\pi}$ in a way which follows the 
same idea as we used in (\ref{eqn:65e}), but where we now separate 
multiple pieces of the product:
\[
\tau_{\pi} = 
\bigl( \prod_{1 \leq h < a_j} \gamma_{\ur(h)} \bigr) 
\cdot \gamma_{\ur(a_j)} \gamma_{\ur( a_{i(q-1)} )} \cdot
\bigl( \prod_{a_{i(q-1)} < h < b_{i(q-1)}} \gamma_{\ur(h)} \bigr)  
\cdot \gamma_{\ur( b_{i(q-1)} )} \gamma_{\ur ( a_{i(q-2)} )} \cdots
\]
\begin{equation}   \label{eqn:65i}
\cdots \gamma_{\ur( b_{i(2)} )} \gamma_{\ur ( a_{i(1)} )} \cdot
\bigl( \prod_{a_{i(1)} < h < b_{i(1)}} \gamma_{\ur(h)} \bigr) \cdot 
\gamma_{\ur( b_{i(1)} )} \gamma_{\ur(b_i)} \cdot
\bigl( \prod_{b_i < h \leq k} \gamma_{\ur(h)}  \bigr).
\end{equation}
The reader should have no difficulty to check that, in (\ref{eqn:65i}):
$\gamma_{\ur( b_{i(1)} )} \gamma_{\ur(b_i)}$ sends $i+1$ to $i(1) + 1$,
then $\gamma_{\ur( b_{i(2)} )} \gamma_{\ur ( a_{i(1)} )}$ sends 
$i(1) + 1$ to $i(2) +1$, and so on until we reach 
$\gamma_{\ur(a_j)} \gamma_{\ur( a_{i(q-1)} )}$ sending 
$i(q-1) + 1$ to $j+1$, while all the products
``$\prod_{h} \gamma_{\ur(h)}$'' do not participate in the action.
Hence overall we get that $\tau_{\pi} (i+1) = j+1$, as we wanted to prove.
An illustration of how things look in this case (with $q \geq 2$ and 
$j \neq 0$) appears in Figure 3.

\vspace{10pt}

\begin{center}
\setlength{\unitlength}{0.7cm}
$\begin{picture}(12,2) \thicklines
   \put(0,0){\line(0,1){1}}
   \put(0,0){\line(1,0){3.2}}
   \put(0,-1.5){$\uparrow$}
   \put(0,-2){$a_j$}
   \put(-1.2,0.5){$\ldots$}
   \put(1,-0.5){\line(0,1){1,5}}
   \put(1,-0.5){\line(1,0){4}}
   \put(1,-1.5){$\uparrow$}
   \put(0.8,-2){$a_{i(2)}$}
   \put(1.8,0.5){$\ldots$}
   \put(3.2,0){\line(0,1){0.6}}
   \put(3.8,0.5){$\ldots$}
   \put(5,-0.5){\line(0,1){1.5}}
   \put(5,-1.5){$\uparrow$}
   \put(4.8,-2.1){$b_{i(2)}$}
   \put(6,-0.5){\line(0,1){1.5}}
   \put(6,-0.5){\line(1,0){4.2}}
   \put(6,-1.5){$\uparrow$}
   \put(6,-2){$a_{i(1)}$}
   \put(10.2,-0.5){\line(0,1){1.5}}
   \put(10.2,-1,5){$\uparrow$}
   \put(10,-2){$b_{i(1)}$}
   \put(6.8,0.5){$\ldots$}
   \put(8,0){\line(0,1){0.6}}
   \put(8,0){\line(1,0){3.2}}
   \put(8.75,0.5){$\ldots$}
   \put(11.2,0){\line(0,1){1}}
   \put(11.2,-1.5){$\uparrow$}
   \put(11.1,-2){$b_i$}
   \put(11.8,0.5){$\ldots$}
   \put(-0.3,1.2){$\gamma_j$}
   \put(0.7,1.2){$\gamma_{i(2)}$}
   \put(4.5,1.2){$\gamma_{i(2)}$}
   \put(5.7,1.2){$\gamma_{i(1)}$}
   \put(9.7,1.2){$\gamma_{i(1)}$}
   \put(10.9,1.2){$\gamma_i$}
\end{picture}$

\vspace{1.8cm}

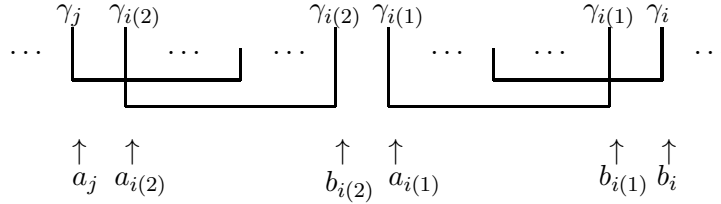
{\bf Figure 3.} {\em An illustration of the case $q=3$,
$j \neq 0$ in the proof of Lemma \ref{lemma:65}.

The picture shows $\gamma_i$ on top of position $b_i$ and
$\gamma_j$ on top of position $a_j$, also

$\gamma_{i(1)}$ on top of positions
$b_{i(1)}$, $a_{i(1)}$ and $\gamma_{i(2)}$ on 
top of positions $b_{i(2)}$, $a_{i(2)}$.

Note that $a_j = a_{i(2)} -1, 
\, b_{i(2)} = a_{i(1)} -1,
\, b_{i(1)} = b_i - 1$.}
\end{center}

\vspace{10pt}

In order to finalize the discussion of the case when $q>1$, 
we must also examine the leftover subcase with $i(q-1) = \ell$ and
$j=0$, when the last equality stated in (\ref{eqn:65h}) amounts to 
$\cyc_{k+1}^{-1} (1) = k+1$.  The way to break into pieces the product
which defines $\tau_{\pi}$ now becomes:
\[
\tau_{\pi} = \gamma_{\ur(1)}
\bigl( \prod_{1 < h < b_{i(q-1)}} \gamma_{\ur(h)} \bigr) 
\cdot \gamma_{\ur( b_{i(q-1)} )} \gamma_{\ur ( a_{i(q-2)} )} \cdots
\]
\begin{equation}   \label{eqn:65j}
\cdots \gamma_{\ur( b_{i(2)} )} \gamma_{\ur ( a_{i(1)} )} \cdot
\bigl( \prod_{a_{i(1)} < h < b_{i(1)}} \gamma_{\ur(h)} \bigr) \cdot 
\gamma_{\ur( b_{i(1)} )} \gamma_{\ur(b_i)} \cdot
\bigl( \prod_{b_i < h \leq k} \gamma_{\ur(h)}  \bigr).
\end{equation}
The steps of the action of $\tau_{\pi}$ on $i+1$ are then the following:
$\gamma_{\ur( b_{i(1)} )} \gamma_{\ur(b_i)}$ sends $i+1$ to $i(1) + 1$,
then $\gamma_{\ur( b_{i(2)} )} \gamma_{\ur ( a_{i(1)} )}$ sends 
$i(1) + 1$ to $i(2) +1$, and so on until we arrive to 
$\gamma_{\ur( b_{i(q-1)} )} \gamma_{\ur( a_{i(q-2)} )}$ sending 
$i(q-2) + 1$ to $i(q-1) + 1 = \ell + 1$, and at the very end 
$\gamma_{\ur(1)}$ sends $\ell + 1$ to $1$.  (Same as was the case in
(\ref{eqn:65i}), the ``$\prod_{h} \gamma_{\ur(h)}$'' pieces listed 
in (\ref{eqn:65j}) do not participate in the action.)  The value of $j$
used in this case is $j=0$; thus the conclusion coming out of 
(\ref{eqn:65j}), that $\tau_{\pi} (i+1) = 1$, completes the proof of
this lemma.
\end{proof}

\vspace{6pt}

\begin{proposition}    \label{prop:66}
Consider the same framework and notation as in the preceding
lemma.  Let $R_1, \ldots , R_p$ be the orbits of 
$\cyc_{k+1} \cdot \sigma_{\pi}$ which intersect $B_{\pi}$.  
Then: $\tau_{\pi}$ has exactly $p$ orbits that are contained
in $\{ 1, \ldots , \ell + 1 \}$, and the sizes of these orbits
are (upon suitable re-ordering) equal to 
$| R_1 \cap B_{\pi} |, \ldots , \ | R_p \cap B_{\pi} |$. 
\end{proposition}

\begin{proof}  We first note that, since
$\cyc_{k+1} \cdot \sigma_{\pi}$ is the inverse of
$\sigma_{\pi} \cdot \cyc_{k+1}^{-1}$, we may view the sets 
$R_1, \ldots , R_p$ as the orbits of 
$\sigma_{\pi} \cdot \cyc_{k+1}^{-1}$ which intersect $B_{\pi}$.
Due to this fact and to how an induced permutation is 
defined (cf.~Notation \ref{def:64}), it is immediate that
$R_1 \cap B_{\pi}, \ldots , R_p \cap B_{\pi}$ can be viewed 
as the orbits of the bijection $\varphi : B_{\pi} \to B_{\pi}$ 
defined by
\begin{equation}   \label{eqn:66a}
\varphi (b) =  \bigl( \left.  ( \sigma_{\pi} \cdot \cyc_{k+1}^{-1} ) 
\right|_{B_{\pi}} \bigr) (b), \ \ \forall \, b \in B_{\pi}.
\end{equation}

On the other hand, the preceding lemma gives the relation
\begin{equation}   \label{eqn:66b}
f \circ \varphi = \psi \circ f
\ \mbox{ (equality of maps from $B_{\pi}$ to $\{1, \ldots , \ell +1 \}$),}
\end{equation}
where $\varphi$ is as in (\ref{eqn:66a}), while
$\psi: \{ 1, \ldots , \ell + 1 \} \to \{ 1, \ldots , \ell + 1 \}$
is defined by
$\psi (i) = \tau_{\pi} (i)$ for $1 \leq i \leq \ell +1$, and
$f : B_{\pi} \to \{ 1, \ldots , \ell + 1 \}$ simply maps  
$b_i$ to $i+1$ for every $0 \leq i \leq \ell$.  When combined 
with the observation from the preceding paragraph, 
(\ref{eqn:66b}) implies that the orbits of the bijection $\psi$ 
are $f( R_1 \cap B_{\pi} ), \ldots , f( R_p \cap B_{\pi} )$.
But the orbits of $\psi$ are precisely the orbits of $\tau_{\pi}$
which are contained in $\{ 1, \ldots , \ell + 1 \}$.  Thus the 
sizes of the latter orbits are 
$| R_1 \cap B_{\pi} |, \ldots , \ | R_p \cap B_{\pi} |$, as 
claimed.
\end{proof}

\vspace{6pt}

The preceding proposition was about the permutation 
$\cyc_{k+1} \cdot \sigma_{\pi} \in S_{k+1}$, which has some 
(possibly) non-trivial orbits contained in 
$\{ 1, \ldots , k+1 \}$ and then fixes every $n > k+1$.  
It is useful to complement the statement of Proposition 
\ref{prop:66} with the following observation.

\vspace{6pt}

\begin{proposition}   \label{prop:67}
The list of orbits $R_1, \ldots, R_p$ mentioned in 
Proposition \ref{prop:66} consists precisely of all the orbits
of $\cyc_{k+1} \cdot \sigma_{\pi}$ that are contained in 
$\{ 1, \ldots , k+1 \}$.
\end{proposition}

\begin{proof} For every $1 \leq j \leq p$, we clearly have that
\[
R_j \cap B_{\pi} \neq \emptyset \ \Rightarrow
\ R_j \cap \{ 1, \ldots , k+1 \} \neq \emptyset
\ \Rightarrow \ R_j \subseteq \{ 1, \ldots , k+1 \}.
\]
It remains to show, conversely, that if $R$ is an orbit of
$\cyc_{k+1} \cdot \sigma_{\pi}$ such that
$R \subseteq \{1 , \ldots , k+1 \}$, then
$R \cap B_{\pi} \neq \emptyset$ (which then implies that $R$
is counted among $R_1, \ldots , R_p$).  We prove
$R \cap B_{\pi} \neq \emptyset$ by displaying a specific 
element of $R$ which is sure to be in $B_{\pi}$, as follows:
\begin{equation}   \label{eqn:67a}
\mbox{ Let $m := \min (R)$.  Then 
$\bigl( \cyc_{k+1} \cdot \sigma_{\pi} \bigr)^{-1} (m) \in B_{\pi}$.}
\end{equation}

For the proof of (\ref{eqn:67a}), we first dispose of the 
case when $m=1$, where we compute that
\[
\bigl( \cyc_{k+1} \cdot \sigma_{\pi} \bigr)^{-1} (1) 
= \sigma_{\pi} \bigl( \, \cyc_{k+1}^{-1} (1) \, \bigr) 
= \sigma_{\pi} (k+1) = k+1 = b_0 \in B_{\pi}.
\]

Suppose now that $R$ has $m := \min (R) > 1$.  
We claim that the number $m-1$ cannot be of the form 
$\max (V_i)$ for one of the blocks 
$V_1, \ldots , V_{\ell}$ of $\pi$.  Indeed, if that would be the
case, then the number 
$m' := \bigl( \cyc_{k+1} \cdot \sigma_{\pi} \bigr)^{-1} (m) \in R$
would come out as 
\[
m' = \sigma_{\pi} ( \cyc_{k+1}^{-1} (m) )
= \sigma_{\pi} (m-1) 
= \sigma_{\pi} ( \max (V_i) )
= \min (V_i);
\]
hence $m' = \min( V_i) \leq \max (V_i) = m-1 <m$, in 
contradiction with the assumption that $m$ is the minimum 
element of $R$.  We conclude that $m-1$ has to be of the form 
$\min (V_i)$ for an $1 \leq i \leq \ell$, when a calculation 
similar to the one above yields the desired conclusion:

\noindent
$\bigl( \cyc_{k+1} \cdot \sigma_{\pi} \bigr)^{-1} (m) 
= \sigma_{\pi} ( \min (V_i) ) = \max (V_i) = b_i \in B_{\pi}$.
\end{proof}

\vspace{6pt}

\begin{example}  \label{example:68}
It is instructive to see what happens when $\pi$ from 
Proposition \ref{prop:66} is a {\em non-crossing pair-partition}. 
This means that $k$ is even, $\pi$ has $\ell$ pairs 
$V_1 = \{ a_1, b_1 \}, \ldots , V_{\ell} = \{ a_{\ell}, b_{\ell} \}$
with $\ell = k/2$ and $a_1 < b_1, \ldots , a_{\ell} < b_{\ell}$,
and it is not possible to find some $1 \leq i,j \leq \ell$ such 
that $a_i < a_j < b_i < b_j$.  It is easy to see that
in this case $\tau_{\pi}$ is the identity permutation -- thus 
when looking at the orbits of $\tau_{\pi}$ that are contained in 
$\{ 1, \ldots , \ell + 1 \}$,  we will find $\ell + 1$ orbits,
each of them of cardinality 1.  Proposition \ref{prop:66} then 
gives us a non-trivial statement about the orbits of 
$\cyc_{k+1} \cdot \sigma_{\pi}$ that are contained in 
$\{ 1, \ldots , k+1 \}$: there are $\ell + 1$ such orbits
$R_1, \ldots , R_{\ell + 1}$, and
$| R_j \cap B_{\pi} | = 1$ for every $1 \leq j \leq \ell + 1$.
\end{example}

$\ $

\section{Moments of the limit law $\muw$, in terms of the 
permutations $\sigma_{\pi}$}

We now resume the discussion about the even moments of the limit
law $\muw$, from where it was left, in Equation (\ref{eqn:56a}) 
at the end of Section 5.  That equation uses the values of the 
character $\vvarphi$ on permutations $\tau_{\pi}$ with 
$\pi \in \cP_{\leq 2} (k)$.  
In this section we show how such a value 
$\vvarphi ( \tau_{\pi} )$  can be described by using the ``other'' 
permutation $\sigma_{\pi}$ associated to $\pi$, and then we record 
(cf.~Proposition \ref{prop:76})
what this has to say concerning the moments of $\muw$. 

\vspace{6pt}

\begin{proposition}   \label{prop:71}
Let $\pi \in \cP_{\leq 2} (k)$ and let 
$\sigma_{\pi}, \tau_{\pi} \in S_{\infty}$ be the permutations 
associated to $\pi$ that were discussed in the preceding sections.
One has
\begin{equation}   \label{eqn:71a}
\vvarphi ( \tau_{\pi} ) = 
\sum_{ \substack{  \ui : \{ 1, \ldots , k+1 \} \to \{ 1, \ldots, d \},   \\
                   \mathrm{constant \ under \ action}  \\
                   \mathrm{of} \ \cyc_{k+1} \cdot \sigma_{\pi} }  }
\ w_{\ui(b_0)} w_{\ui(b_1)} \cdots w_{\ui (b_{\ell})},
\end{equation}
where $k+1 = b_0 > b_1 > \cdots > b_{\ell} \geq 1$ and the cycle
$\cyc_{k+1} \in S_{k+1}$ are as in Proposition \ref{prop:66}, while
the weights $w_1 \geq w_2 \geq \cdots \geq w_d > 0$ are as in
Notation \ref{def:22}.
\end{proposition}

\begin{proof} Same as in Propositions \ref{prop:66} and 
\ref{prop:67}, let $R_1, \ldots , R_p$ be the orbits of
$\cyc_{k+1} \cdot \sigma_{\pi}$ which are contained in $\{ 1, \ldots , k+1 \}$.
All the tuples $\ui : \{ 1, \ldots , k+1 \} \to \{ 1, \ldots ,d \}$ 
considered on the right-hand side of (\ref{eqn:71a}) are obtained, without 
repetitions, by starting with an arbitrary tuple 
$\uj : \{ 1, \ldots , p \} \to \{ 1, \ldots , d\}$ and by putting
\[
\ui (h) = \uj(1) \mbox{ for } h \in R_1, \  \ldots ,
\ \ui (h) = \uj(p) \mbox{ for } h \in R_p.
\]
It follows that the right-hand side of (\ref{eqn:71a}) is equal to
\[
\sum_{  \uj : \{ 1, \ldots , p \} \to \{ 1, \ldots, d \} }
\ w_{\uj(1)}^{ |R_1 \cap B_{\pi}| } 
\cdots w_{\uj(p)}^{ |R_p \cap B_{\pi}| },
\]
where $B_{\pi} := 
\{ b_{\ell}, \ldots , b_1, b_0 \} \subseteq \{ 1, \ldots , k+1 \}$.
Moreover, the latter sum clearly factors as
\begin{equation}    \label{eqn:71b}
\Bigl( \, \sum_{j=1}^d 
\ w_j^{ |R_1 \cap B_{\pi}| } \, \Bigr)
\, \cdots \, \Bigl( \, \sum_{j=1}^d  w_j^{ |R_p \cap B_{\pi}| } \, \Bigr).
\end{equation}

Now, Proposition \ref{prop:66} assures us that the the orbits 
of $\tau_{\pi}$ that are contained in $\{ 1, \ldots , \ell + 1 \}$ can
be listed as $Q_1, \ldots , Q_p$, in such a way that we have
$| Q_i | = | R_i \cap B_{\pi} |$ for every $1 \leq i \leq p$. Since 
$\tau_{\pi}$ fixes every $m > \ell + 1$, the definition of $\vvarphi$
gives 
\[
\vvarphi ( \tau_{\pi} ) 
= \llambda_{|Q_1|} \cdots \llambda_{|Q_p|}
= \Bigl( \, \sum_{j=1}^d \ w_j^{ |Q_1| } \, \Bigr)
\, \cdots \,  \Bigl( \, \sum_{j=1}^d  w_j^{ |Q_p| } \, \Bigr).
\]
Upon comparing this with the product indicated in (\ref{eqn:71b}),
we see that $\vvarphi ( \tau_{\pi} )$ is indeed equal to the 
right-hand side of (\ref{eqn:71a}), as claimed.
\end{proof}

\vspace{6pt}

We next note that the formula (\ref{eqn:71a}) obtained above
can be re-written in a way which doesn't explicitly refer to $k+1$
-- this goes by replacing the occurrence of $\cyc_{k+1}$ on the 
right-hand side with an occurrence of the cycle 
$\cyc_k = (1, \ldots , k) \in S_k \subseteq S_{\infty}$.

\vspace{6pt}

\begin{corollary}   \label{cor:72}
In the same framework and notation as in Proposition \ref{prop:71},
one has
\begin{equation}   \label{eqn:72a}
\vvarphi ( \tau_{\pi} ) = 
\sum_{ \substack{  \ui : \{ 1, \ldots ,k \} \to \{ 1, \ldots , d \},  \\
                   \mathrm{constant \ under \ action}  \\
                   \mathrm{of} \ \cyc_k \cdot \sigma_{\pi} }  }
\ w_{\ui(1)} \cdot 
\bigl( w_{\ui(b_1)} \cdots w_{\ui (b_{\ell}) } \bigr) .
\end{equation}
\end{corollary}

\begin{proof} Let us denote
\[
\cI := \bigl\{ \ui : \{ 1, \ldots ,k+1 \} \to \{ 1, \ldots , d \}
\mid \ui \mbox{ is constant under the action of}
\ \cyc_{k+1} \cdot \sigma_{\pi} \bigr\},
\mbox{ and}
\]
\[
\cJ := \bigl\{ \uj : \{ 1, \ldots ,k \} \to \{ 1, \ldots , d \}
\mid \uj \mbox{ is constant under the action of} 
\ \cyc_k \cdot \sigma_{\pi} \bigr\}. 
\]
Moreover, for every 
$\ui : \{ 1, \ldots ,k+1 \} \to \{ 1, \ldots , d \}$ we will use 
the notation ``$r( \ui )$'' for the restriction of $\ui$ to 
$\{ 1, \ldots , k \}$. 

For the partition $\pi \in \cP_{\leq 2} (k)$ we are working with,
recall that we denote $a_1 = \min (V_1)$, where $V_1$ is the block of
$\pi$ which contains the number $k$.  That is: either $V_1$ is a singleton
block at $k$ and $a_1 = b_1 =k$, or $V_1$ is a pair $\{a_1, b_1\}$
with $a_1 < b_1 = k$.  Direct inspection shows that the permutation 
$\cyc_{k+1} \circ \sigma_{\pi}$ sends $a_1 \mapsto k+1 \mapsto 1$, while
$\cyc_k \circ \sigma_{\pi}$ sends directly $a_1 \mapsto 1$. (This includes the 
possibility that $a_1 =1$, when $\cyc_{k+1} \circ \sigma_{\pi}$ swaps 
$1$ with $k+1$, while $\cyc_k \circ \sigma_{\pi}$ has $1$ as a fixed point.)
As a consequence, every $\ui \in \cI$ has $\ui (a_1) = \ui (k+1) = \ui (1)$,
and every $\uj \in \cJ$ has $\uj (a_1) = \uj (1)$.  A further consequence 
of this observation is that the following statements (a) and (b) are holding.

\vspace{6pt}

\noindent
(a)  If $\ui \in \cI$, then $r( \ui ) \in \cJ$.  It thus 
makes sense to consider the map $r : \cI \to \cJ$.

\vspace{6pt}

\noindent
(b)  The restriction map $r : \cI \to \cJ$ defined 
in (a) is bijective.

\vspace{6pt}

The verifications needed in order to obtain (a) and (b) are
straightforward.  For instance: for checking that the map $r$ 
is surjective we start with a tuple $\uj \in \cJ$ and check that
\[
\ui (h) := \left\{   \begin{array}{ll}
\uj (h), & \mbox{ if $1 \leq h \leq k$} \\
\uj (1), & \mbox{ if $h=k+1$} 
\end{array}   \right\} \mbox{ for $1 \leq h \leq k+1$}
\]
defines a tuple $\ui \in \cI$ such that $r( \ui ) = \uj$.  We 
leave the verification details (for this and for the other facts 
that come up when checking (a)+(b)) as an exercise to the reader.

Now let us use the bijection $r : \cI \to \cJ$ as a change of variable
in the summation on the right-hand side of (\ref{eqn:71a}).  The term
indexed by $\ui \in \cI$ in the said summation is
\[
w_{\ui(b_0)} \cdot  \bigl(
w_{\ui(b_1)} \cdots w_{\ui (b_{\ell})} \bigr)
= w_{\ui(1)} \cdot  \bigl(
w_{\ui(b_1)} \cdots w_{\ui (b_{\ell})} \bigr)
= w_{\uj(1)} \cdot  \bigl(
w_{\uj(b_1)} \cdots w_{\uj (b_{\ell})} \bigr),
\]
where at the first equality sign we used the fact that
$\ui (b_0) = \ui (k+1) = \ui (1)$, and at the second equality 
sign we put $\uj = r( \ui ) \in \cJ$.  Thus the change of variable 
provided by $r$ converts the summation from the right-hand side of
(\ref{eqn:71a}) into the summation on the right-hand side of 
(\ref{eqn:72a}), as required.
\end{proof}

\vspace{6pt}

\begin{remark}    \label{rem:73}
The next corollary records what comes out when the expression for 
$\vvarphi ( \tau_{\pi} )$ that was just obtained is replaced back 
in the formula (\ref{eqn:56a}) for an even moment of the 
limit law $\muw$.  When stating the corollary, we will make two adjustments
in the notation.

\vspace{6pt}

\noindent
$1^o$ The fact that $\ui$ is constant under the action of 
$\cyc_k \cdot \sigma_{\pi}$ can be written for short as
\begin{equation}   \label{eqn:73a}
\ui \circ ( \cyc_k \cdot \sigma_{\pi} ) = \ui, 
\end{equation}
with the slight abuse of notation that 
$\cyc_k \cdot \sigma_{\pi} \in S_k \subseteq S_{\infty}$ is now treated
as a function from $\{ 1, \ldots , k \}$ to itself (while the tuple 
$\ui$ is viewed as a function from $\{ 1, \ldots ,k \}$ to 
$\{ 1, \ldots , d \}$).  We note, moreover, that upon composing
with $\sigma_{\pi}$ on the right in (\ref{eqn:73a}) and upon taking
into account that $\sigma_{\pi}^2$ is the identity permutation, we 
can equivalently re-write (\ref{eqn:73a}) as 
\begin{equation}   \label{eqn:73b}
\ui \circ \cyc_k = \ui \circ \sigma_{\pi}. 
\end{equation}

\vspace{3pt}

\noindent
$2^o$  In order to better keep in mind that the numbers 
$b_1, \ldots , b_{\ell}$ depend on $\pi$, we will re-write the product
$w_{\ui (b_1)} \cdots w_{\ui (b_{\ell}) }$ that appeared in 
(\ref{eqn:72a}) in the form
$\prod_{V \in \pi} w_{ \ui ( \max (V) ) }$.
\end{remark}

\vspace{6pt}

\begin{corollary}   \label{cor:74}
For an even $k \in \bN$, the moment
$\int_{\bR} t^k \, d \muw (t)$ is equal to
\begin{equation}   \label{eqn:74a}
\sum_{\ui : \{ 1, \ldots , k \} \to \{ 1, \ldots , d \} } 
\Bigl( \sum_{ \substack{  \pi \in \cP_{\leq 2} (k), \\
          \mathrm{with} \ \ui \circ \cyc_k = \ui \circ \sigma_{\pi} } }
(-1)^{|\pi_1 |/2} \cdot ( |\pi|_1 - 1)!! \cdot
w_{\ui (1)} \cdot \prod_{V \in \pi}  w_{\ui (\max (V))} \Bigr) .
\end{equation}
\end{corollary}

\begin{proof}  Upon plugging the formula (\ref{eqn:72a}) into
(\ref{eqn:56a}), we find $\int_{\bR} t^k \, d \mu (t)$ written 
as a double sum indexed by
\[
\Bigl\{ ( \pi, \ui ) \ \vline \ \pi \in \cP_{\leq 2} (k) 
\mbox{ and } \ui : \{ 1, \ldots , k \} \to \{ 1, \ldots , d \},
\mbox{ such that } 
\ui \circ \cyc_k = \ui \circ \sigma_{\pi} \Bigr\}.
\]
In (\ref{eqn:74a}) we simply record this double sum, in the 
form of an iterated sum.
\end{proof}

\vspace{6pt}

For a further bit of processing of the moment formula obtained 
in (\ref{eqn:74a}), we observe that the double factorial appearing
in that formula can be made to disappear, if we do the following
trick: merge in pairs the singleton blocks of the partition $\pi$,
and use a colouring of the pairs of the resulting pair-partition.
We thus proceed as follows.

\vspace{6pt}

\begin{notation-and-remark}    \label{def:75}
{\em (Bicoloured pair-partitions.)}
$1^o$ We will use the name {\em bicoloured pair-partition}
for a pair-partition $\rho$ of $\{ 1, \ldots , k\}$ where every 
pair of $\rho$ was painted in either blue or red.  The set of all
bicoloured pair-partitions of $\{ 1, \ldots , k \}$ will be denoted
by $\cP_2^{(b-r)} (k)$ (with the usual
convention that 
$\cP_2^{(b-r)} (k) = \emptyset$
for $k$ odd).

\vspace{3pt} 

\noindent
$2^o$ For $k$ even, we will use a ``red-pair-breaking'' map
$\cP_2^{(b-r)} (k) \ni \rho \mapsto \pi \in \cP_{\leq 2} (k)$
described as follows: $\pi$ is obtained by taking every red block
of $\rho$ and breaking it into two singleton blocks.  It is useful 
to note that, for every $\pi \in \cP_{\leq 2} (k)$, the pre-image
of $\pi$ under this red-pair-breaking map has cardinality equal to 
$( | \pi |_1 - 1)!!$.  The double factorial appears because 
we are counting the number of ways of grouping the $| \pi |_1$ 
singleton blocks of $\pi$ into $|\pi|_1 /2$ red pairs of a 
bicoloured pair-partition $\rho$.

\vspace{3pt}

\noindent
$3^o$ For $k$ even and $\rho \in \cP_2^{(b-r)} (k)$
we will denote
$\sigma^{\mathrm{(blue)}}_{\rho} := 
\prod_{ \substack{  \{ a, b \} \ \mathrm{blue}  \\
        \mathrm{pair} \ \mathrm{of} \ \rho } } \ (a,b) \in S_k$.
\end{notation-and-remark}

\vspace{6pt}

The formula obtained in Corollary \ref{cor:74} is then 
further processed as follows.

\vspace{6pt}

\begin{proposition}   \label{prop:76}
For every $k \in \bN$, the moment
$\int_{\bR} t^k \, d \muw (t)$ is equal to
\begin{equation}   \label{eqn:76a}
\sum_{\ui : \{ 1, \ldots , k \} \to \{ 1, \ldots , d \} }
w_{\ui (1)}
\Bigl(  \sum_{ \substack{ \rho \in \cP_2^{(b-r)} (k), \\
               \mathrm{with} \ \ui \circ \cyc_k
                     = \ui \circ \sigma^{\mathrm{blue}}_{\rho} } }
\, \prod_{ \substack{ \{ p,q \} \in \rho, \\ 
                  \mathrm{blue,} \ p<q } }  w_{\ui (q)} \cdot
\prod_{ \substack{ \{ p,q \} \in \rho, \\ \mathrm{red,} \ p<q } }
(- w_{\ui (p)} w_{\ui (q)} ) \,  \Bigr) .
\end{equation}
\end{proposition}

\begin{proof}  For $k$ odd, both quantities involved are equal
to $0$.  For $k$ even, the expression in (\ref{eqn:76a}) is obtained 
from the one in (\ref{eqn:74a}) when we use the surjection 
$\cP_2^{(b-r)} (k) \to \cP_{\leq 2} (k)$ indicated in 
Remark \ref{def:75}.2.
\end{proof}

$\ $

\section{Wick-style formulas and the CCR-GUE model 
for the law $\muw$}

In this section we examine the CCR-GUE matrix $M$ that was 
introduced in Definition \ref{def:28}, and we give the proof 
of the main result of the paper, Theorem \ref{thm:29}.  This
theorem claims that the law of $M$ (considered in the natural
$*$-probability space where $M$ lives) coincides with the 
limit law $\muw$ obtained in Theorem \ref{thm:25}.  

The next remark gives an outline of how the section is 
organised.

\vspace{6pt}

\begin{remark}    \label{rem:81}
Recall from Definition \ref{def:28} that the entries 
$\{ a_{i,j} \mid 1 \leq i,j \leq d \}$ of $M$ are for the most
part commuting, as they are drawn from a family of commuting 
independent subalgebras 
\begin{equation}    \label{eqn:81a}
\{ \cA_o \} \cup \{ \cA_{i,j} \mid 1 \leq i < j \leq d \}
\end{equation}
of $\cA$.  The notion of ``commuting independent subalgebras'' 
is reviewed in Definition \ref{def:82}.1 below.  Note that this 
notion does not require an individual subalgebra $\cA_{i,j}$ 
to be commutative -- and in fact $\cA_{i,j}$ will surely not be
commutative if the weights $w_i, w_j$ are distinct.

Concerning the algebra $\cA_o \subseteq \cA$, where we pick the 
diagonal entries $a_{1,1}, \ldots , a_{d,d}$ of $M$: this one can
be assumed to be commutative, and the formula used for computing 
joint moments of $a_{1,1}, \ldots , a_{d,d}$ is the usual Wick 
formula for real Gaussian random variables. We review this Wick
formula in Definition \ref{def:82}.2.

On the other hand: for $1 \leq i < j \leq d$ we will derive a
CCR-analogue of the Wick-formula for a complex Gaussian random 
variable, which allows us to compute the $*$-moments of 
$a_{i,j} \in \cA_{i,j}$.  The precise statement of how this goes 
is given in Proposition \ref{prop:83}.  

Together with the commuting independence of the subalgebras from
(\ref{eqn:81a}), the two Wick formulas (usual version and CCR-version)
will then give us a Wick-style summation formula for a general joint
moment of the full family of entries $a_{i,j}$ of $M$.  This formula
is derived in Proposition \ref{prop:84}.  

Finally, the proof of Theorem \ref{thm:29} is easily obtained by
comparing the moments of the law of $M \in M_d ( \cA )$ (as they 
come out when we use the Wick-style formula of 
Proposition \ref{prop:84}) against the description of moments of 
$\muw$ that was obtained in Proposition \ref{prop:76}.
\end{remark}

\vspace{6pt}

\begin{definition}     \label{def:82}
{\em (Review of some basic notions needed below.)}

\noindent
$1^o$ {\em Commuting independent subalgebras.}
Let $( \cA , \varphi )$ be a $*$-probability space and let 
$( \cA_{\lambda} )_{\lambda \in \Lambda}$ be a family of 
unital $*$-subalgebras of $\cA$.  We say that the $\cA_{\lambda}$'s
are {\em commuting independent} to mean that the following two 
conditions (i) + (ii) are fulfilled:

\vspace{3pt}

(i) For every $\lambda_1 \neq \lambda_2$ in $\Lambda$ and every 
$x_1 \in \cA_{\lambda_1}$, $x_2 \in \cA_{\lambda_2}$, one has that
$x_1 x_2 = x_2 x_1$.

\vspace{3pt}

(ii) Whenever $\lambda_1, \ldots , \lambda_k \in \Lambda$ are distinct
($\lambda_i \neq \lambda_j$ for $1 \leq i < j \leq k$), one has
\[
\varphi ( x_1 x_2 \cdots x_k) = \varphi (x_1) \varphi (x_2) 
\cdots \varphi (x_k), \ \ \forall \, x_1 \in \cA_{\lambda_1}, 
\ldots , x_k \in \cA_{\lambda_k}.
\]

\vspace{3pt}

\noindent
$2^o$ {\em Gaussian family of selfadjoints, via the Wick formula.}
Let $(\cA , \varphi )$ be a $*$-probability space where the 
algebra $\cA$ is commutative.  Suppose we are given a family
$x_1, \ldots , x_d$ of selfadjoint elements of $\cA$, and 
a symmetric matrix $C = [c_{i,j}]_{i,j=1}^d \in M_d ( \bR )$.
We say that $x_1, \ldots , x_d$ form a 
{\em centred Gaussian family with covariance matrix $C$} 
to mean that the following formula for computation of joint 
moments is holding: for every $k \geq 1$ and every tuple 
$\ui : \{ 1, \ldots , k \} \to \{ 1, \ldots , d \}$, one has
\begin{equation}    \label{eqn:82a}
\varphi ( x_{\ui (1)} \cdots x_{\ui (k)} )
= \sum_{\rho \in \cP_2 (k)}  \prod_{\{ p,q \} \in \rho}
c_{\ui (p), \ui (q)},
\end{equation}
with the usual convention that for $k$ odd the right-hand 
side of (\ref{eqn:82a}) is to be read as ``$0$''.
\end{definition}

\vspace{6pt}

\begin{proposition}   \label{prop:83}
{\em (Wick's lemma for a CCR-complex-Gaussian element.)}

\noindent
Let $( \cA , \varphi )$ be a $*$-probability space and suppose 
that $a \in \cA$ is a centred CCR-complex-Gaussian element
with parameters $\wx_{(1,*)}, \wx_{(*,1)} \in (0, \infty)$,
in the sense of Definition \ref{def:26}.
This means that $a$ and $a^{*}$ satisfy the commutation relation
(\ref{eqn:26a}), and that the expectations 
$\varphi ( a^p (a^{*})^q )$ are evaluated according to the formula
(\ref{eqn:26b}) from that definition.  We also put,
for convenience, $\wx_{(1,1)} = \wx_{(*,*)} = 0$.

\noindent
Then, for every $k \in \bN$ and 
$( \ee (1), \ldots , \ee (k) ) \in \{ 1,* \}^k$, one has
\begin{equation}    \label{eqn:83a}
\varphi \bigl( a^{\ee (1)} \cdots a^{\ee (k)} \bigr) 
= \sum_{\rho \in \cP_2 (k)}
\prod_{ \substack{ \{ p,q \} \in \rho,  \\  \mathrm{with} \ p<q} }
\wx_{( \ee (p), \ee (q))},
\end{equation}
with the usual convention that for $k$ odd the right-hand 
side of (\ref{eqn:83a}) is to be read as ``$0$''.
\end{proposition}

\begin{proof}  We will proceed by induction on the number of 
``$(*,1)$-inversions'' $\mathrm{inv} ( \uee )$ in a tuple 
$\uee \in \sqcup_{k=1}^{\infty} \{ 1,* \}^k$, where for 
$k \geq 1$ and 
$\uee = ( \ee (1), \ldots, \ee (k) ) \in \{ 1,* \}^k$
we denote
\[
\mathrm{inv} ( \uee ) := \ \vline \
\bigl\{ (p,q) \mid 1 \leq p < q \leq k, \, \ee (p) = *,
\, \ee (q) = 1 \bigr\} \ \vline \, .
\]

\vspace{6pt}

The base-case of the induction is the one when
$\mathrm{inv} ( \uee ) = 0$.  This refers to the situation 
when there is no instance of a $*$ preceding a $1$ among the
components of $\uee$, hence when $\uee$ is of the form 
$(1, \ldots , 1, *, \ldots , *)$.  In this case, the left-hand
side of (\ref{eqn:83a}) is of the form 
$\varphi ( a^u \, (a^{*})^v )$ for some $u,v \geq 0$ with 
$u+v = k$, and is therefore explicitly described by 
Equation (\ref{eqn:26b}).  It is an easy counting exercise,
left to the reader, to check that in this special situation 
the right-hand side of Equation (\ref{eqn:83a}) also matches 
the formula indicated in Equation (\ref{eqn:26b}). 

\vspace{6pt}

For the rest of the proof we work on the induction step of the
argument.  We thus fix an $\ell \geq 1$, we assume that 
(\ref{eqn:83a}) holds for all the tuples 
$\uee \in \sqcup_{k=1}^{\infty} \{ 1,* \}^k$ which have 
$\mathrm{inv} ( \uee ) \leq \ell -1$, and we prove that it also 
holds for tuples $\uee$ which have $\mathrm{inv} ( \uee ) = \ell$.
To this end, we consider a tuple
$\uee_o = ( \ee_o (1), \ldots, \ee_o (k) )$
with $\mathrm{inv} ( \uee_o ) = \ell$.
This particular $\uee_o$ is also fixed for the rest of the 
proof, and our goal is to verify that (\ref{eqn:83a}) holds 
for it.  We will run the verification by assuming that
the length $k$ of $\uee_o$ is such that $k \geq 3$ (if $k \leq 2$,
then (\ref{eqn:83a}) follows from immediate calculations of 
moments of order 1 and 2).

Since $\mathrm{inv} ( \uee_o ) = \ell > 0$, there has to exist a 
$j \in \{ 1, \ldots , k-1 \}$ such that $\ee_o (j) = *$ and 
$\ee_o (j+1) = 1$.  That is, the monomial 
\begin{equation}    \label{eqn:83b}
a^{\uee_o} := a^{\ee_o (1)} \cdots a^{\ee_o (k)}
\end{equation}
has an $a^{*}$ on position $j$ which is immediately followed by an 
$a$ on position $j+1$.  Let $\uee$ be the tuple in $\{ 1,* \}^k$ 
which is obtained by swapping the positions $j$ and $j+1$ of $\ee_o$:
\[
\ee (j) = 1, \ \ee (j+1) = *, \mbox{ and }
\ee (p) = \ee_o (p) \mbox{ for all }
p \in \{1, \ldots ,k \} \setminus \{ j, j+1 \}.
\]
We notice that:
\begin{align*}
a^{\ee_o} 
& = (a^{\ee (1)} \cdots a^{\ee (j-1)} ) (a^{*}a )
    (a^{\ee (j+2)} \cdots a^{\ee (k)} )              \\
& = (a^{\ee (1)} \cdots a^{\ee (j-1)} ) 
    (a a^{*} + (\wx_{(*,1)} - \wx_{(1,*)} ) \oneA )
    (a^{\ee (j+2)} \cdots a^{\ee (k)} )               \\
& = a^{\ee} + (\wx_{(*,1)} - \wx_{(1,*)} ) a^{\ee '},
\end{align*}
where $\ee ' := 
( \ee_o (1), \ldots , \ee_o (j-1), \ee_o (j+2), \ldots , \ee_o (k) )
\in \{ 1,* \}^{k-2}$ is obtained by removing the components 
$j$ and $j+1$ out of $\ee_o$, and where $a^{\ee}, a^{\ee '}$ are 
defined in the same way as $a^{\ee_o}$ was defined in 
(\ref{eqn:83b}).  We hence conclude that:
\begin{equation}     \label{eqn:83c}
\varphi ( a^{\ee_o} ) = 
\varphi ( a^{\ee} ) 
+ (\wx_{(*,1)} - \wx_{(1,*)} ) \varphi ( a^{\ee '} ).
\end{equation}

Now, it is immediately verified that 
$\mathrm{inv} ( \ee ) = \mathrm{inv} ( \ee_o ) - 1 = \ell -1$, and also
that $\mathrm{inv} ( \ee ' ) < \mathrm{inv} ( \ee_o ) = \ell$; hence
the induction hypothesis applies to both $\ee$ and $\ee '$, and assures
us that we have $\varphi ( a^{\ee} ) = \Sigma ( \ee )$ and
$\varphi ( a^{\ee '} ) = \Sigma ( \ee ' )$, where we denote:
\[
\Sigma ( \ee )  := \sum_{\rho \in \cP_2 (k)}
\prod_{ \substack{ \{ p,q \} \in \rho,  \\  \mathrm{with} \ p<q} }
\wx_{( \ee (p), \ee (q))},
\ \ \Sigma ( \ee' )  := \sum_{\rho ' \in \cP_2 (k-2)}
\prod_{ \substack{ \{ p,q \} \in \rho ',  \\  \mathrm{with} \ p<q} }
\wx_{( \ee' (p), \ee' (q))}.
\]

Recall that what we need to prove is the equality
$\varphi ( a^{\ee_o} ) = \Sigma ( \ee_o )$, with
\[
\Sigma ( \ee_o )  := \sum_{\rho \in \cP_2 (k)}
\prod_{ \substack{ \{ p,q \} \in \rho,  \\  \mathrm{with} \ p<q} }
\wx_{( \ee_o (p), \ee_o (q))}.
\]
In view of (\ref{eqn:83c}), the desired equality 
$\varphi ( a^{\ee_o} ) = \Sigma ( \ee_o )$ will follow if we can verify
that
\begin{equation}     \label{eqn:83d}
\Sigma ( \ee_o ) = 
\Sigma ( \ee ) + (\wx_{(*,1)} - \wx_{(1,*)} ) \Sigma ( \ee ' ).
\end{equation}

So we are left to establish that (\ref{eqn:83d}) holds.  To that end,
let us also consider the sum
\begin{equation}    \label{eqn:83e}
\widetilde{\Sigma} =  \sum_{ \substack{ \rho \in \cP_2 (k) ,  \\
\{ j, j+1 \} \ \mathrm{not \ a \ pair \ in} \ \rho } }
\ \ \prod_{ \substack{ \{ p,q \} \in \rho,  \\  \mathrm{with} \ p<q} }
\wx_{( \ee_o (p), \ee_o (q))}.
\end{equation}
We make the observation that:
\begin{equation}   \label{eqn:83f}
\Sigma ( \ee_o ) - \widetilde{\Sigma} 
= \wx_{(*,1)} \cdot \Sigma ( \ee ' ) \ \mbox{ and }
\  \Sigma ( \ee ) - \widetilde{\Sigma} 
= \wx_{(1,*)} \cdot \Sigma ( \ee ' ).
\end{equation}
Indeed, for the first Equation (\ref{eqn:83f}) we write the 
difference $\Sigma ( \ee_o ) - \widetilde{\Sigma}$ as
\begin{align*}
\sum_{ \substack{ \rho \in \cP_2 (k) ,  \\ \{ j, j+1 \} \in \rho } }
\ \ \prod_{ \substack{ \{ p,q \} \in \rho,  \\  \mathrm{with} \ p<q} }
\wx_{( \ee_o (p), \ee_o (q))}
& = \wx_{( \ee_o (j), \ee_o (j+1) )} \cdot
\sum_{\rho ' \in \cP_2 (k-2)}
\prod_{ \substack{ \{ p,q \} \in \rho ',  \\  \mathrm{with} \ p<q} }
\wx_{( \ee' (p), \ee' (q))}    \\
& = \wx_{(*,1)} \cdot \Sigma ( \ee' ),
\end{align*}
where at the first equality sign we let the running pair-partition
$\rho'$ be obtained out of $\rho$ by removing its pair
$\{ j,j+1 \}$ and then redenoting the elements of 
$\{ 1, \ldots ,k \} \setminus \{ j,j+1 \}$ as 
$1, \ldots , k-2$.  The second Equation (\ref{eqn:83f}) is obtained
in a similar manner, where now we view $\ee'$ as the restriction of 
$\ee$, and the factor $\wx_{(\ee (j), \ee (j+1))}$ that has to be 
treated separately is equal to $\wx_{(1,*)}$.

Finally, subtracting the second Equation (\ref{eqn:83f}) out 
of the first leads precisely to the formula (\ref{eqn:83d})
that we had been left to prove.
\end{proof}

\vspace{6pt}

Now we put together the Wick formulas described above, 
to get a formula for the 
joint moments of the entries of the CCR-GUE matrix.

\vspace{6pt}

\begin{proposition}   \label{prop:84}
Let $M = [ a_{i,j} ]_{i,j=1}^d \in M_d ( \cA )$ be the
CCR-GUE matrix introduced in Definition \ref{def:28}.
Then for every $k \in \bN$ and 
$\ui, \uj : \{ 1, \ldots , k \} \to \{ 1, \ldots , d \}$ 
we have
\begin{equation}   \label{eqn:84a}
\varphi \bigl(  a_{\ui (1), \uj (1)} 
\cdots a_{\ui (k), \uj (k)} \bigr)
\end{equation}
\[
=  \sum_{\rho \in \cP_2^{\mathrm{(b-r)}} (k)}
\Bigl( \, \prod_{ \substack{ \{ p,q \} \in \rho, \\ 
                  \mathrm{blue,} \ p<q } }
\delta_{\ui (p), \uj (q)} \delta_{\ui (q), \uj (p)} w_{\ui (q)}
\cdot
\prod_{ \substack{ \{ p,q \} \in \rho, \\ \mathrm{red,} \ p<q } }
\delta_{\ui (p), \uj (p)} \delta_{\ui (q), \uj (q)} 
\cdot (- w_{\ui (p)} w_{\ui (q)} ) \, \Bigr) ,
\]
where $\cP_2^{\mathrm{(b-r)}} (k)$ is the set of bicoloured 
pair-partitions from Notation \ref{def:75}.
\end{proposition}

\begin{proof} Throughout the whole proof we fix a $k \in \bN$ and two
tuples $\ui, \uj : \{ 1, \ldots , k \} \to \{ 1, \ldots , d \}$, for 
which we will verify that (\ref{eqn:84a}) holds.  It will come in 
handy to use the following notation:
\begin{equation}    \label{eqn:84c}
\left\{   \begin{array}{ll}
\bullet & \mbox{Let $P_o := \{ 1 \leq p \leq k \mid \ui (p) = \uj (p) \}$. } \\
        &                                                                    \\
\bullet & \mbox{For every $1 \leq u < v \leq d$, let}   \\
        & \mbox{ $\ $ $P_{u,v} := \{ 1 \leq p \leq k \mid 
         ( \ui (p) , \uj (p) ) = (u,v)$ or $( \ui (p) , \uj (p) ) = (v,u) \}$.}
\end{array}  \right.
\end{equation}
It is immediate that the sets introduced in (\ref{eqn:84c}) are pairwise 
disjoint (with the possibility that some of them are empty), and their 
union is equal to $\{ 1, \ldots , k \}$.  We will run the calculations 
below by making the assumption that all these sets have even 
cardinality.  We leave it as an exercise to the reader 
to verify that if this was not the case (i.e.~either $|P_o|$ or one of
the $|P_{u,v}|$'s was odd), then calculations similar to those shown below 
would quickly end in the conclusion that both sides of (\ref{eqn:84a}) are 
equal to $0$.  Also: we will write the remaining part of this proof by 
assuming we have $P_o \neq \emptyset$.  For the case when $P_o = \emptyset$, 
one simply has to ignore all the considerations pertaining to $P_o$ 
(e.g.~the discussion about the quantities $C_o$ and $C_o'$) which appear
below.

\vspace{6pt}

On the left-hand side of (\ref{eqn:84a}), the hypothesis (i) of commuting 
independence entails a factorization
\begin{equation}    \label{eqn:84d}
\varphi ( a_{\ui (1), \uj (1)} \cdots a_{\ui (k), \uj (k)} )
= C_o \cdot \prod_{ \substack{ 1 \leq u < v \leq d \\
               \mathrm{with} \ P_{u,v} \neq \emptyset } } C_{u,v},
\end{equation}
where the quantities $C_o$ and $C_{u,v}$ are given by
\begin{equation}    \label{eqn:84e}
C_o = \varphi \bigl( \, \prod_{p \in P_o} a_{\ui (p), \ui (p)} \bigr)
\ \mbox{ and }
\ C_{u,v} = \varphi \bigl( \prod_{p \in P_{u,v}} a_{\ui (p), \uj (p)} \bigr).
\end{equation}
We note that $\prod_{p \in P_o} a_{\ui (p), \ui (p)}$ is a commuting
product and that $C_o$ can be evaluated by using the Wick formula 
reviewed in Definition \ref{def:82}.2.  A product 
$\prod_{p \in P_{u,v}} a_{\ui (p), \uj (p)}$ is non-commuting, hence
it is important to mention that its factors are written in the increasing
order of the elements of $P_{u,v}$; each of these factors is either an 
$a_{u,v}$ or an $a_{u,v}^{*}$, and $C_{u,v}$ can thus be evaluated by 
using the CCR-Wick formula from Proposition \ref{prop:83}.

\vspace{6pt}

We now turn to the right-hand side of (\ref{eqn:84a}), which we re-write
more concisely by making the following ad-hoc definition: we say that 
$\rho \in \cP_2^{\mathrm{(b-r)}} (k)$ is 
{\em compatible with $\ui$ and $\uj$} when (j) + (jj) below are holding:
\begin{equation}    \label{eqn:84f}
\left\{   \begin{array}{l}
\mbox{(j) For every blue pair $\{ p,q \}$ of $\rho$ 
      one has $\ui (p) = \uj (q)$ and $\ui (q) = \uj (p)$;}    \\
                    \\
\mbox{(jj) For every red pair $\{ p,q \}$ of $\rho$ 
      one has $\ui (p) = \ui (q)$ and $\uj (p) = \uj (q)$.}    \\
\end{array}  \right.
\end{equation}
The conditions (j) + (jj) from (\ref{eqn:84f}) are precisely addressing
the Kronecker deltas on the right-hand side of (\ref{eqn:84a}), and the
expression written there takes the form:
\begin{equation}    \label{eqn:84g}
\sum_{\substack{ \rho \in \cP_2^{\mathrm{(b-r)}} (k) \\
                 \mathrm{compatible \ with} \ \ui, \uj } }
\Bigl( \, \prod_{ \substack{ \{ p,q \} \in \rho, \\ 
                  \mathrm{blue,} \ p<q } } w_{\ui (q)} \cdot
\prod_{ \substack{ \{ p,q \} \in \rho, \\ \mathrm{red,} \ p<q } }
(- w_{\ui (p)} w_{\ui (q)} ) \, \Bigr) .
\end{equation}

We next observe (via direct comparison between (\ref{eqn:84c}) and
(\ref{eqn:84f})) that if $\rho \in \cP_2^{\mathrm{(b-r)}} (k)$
is compatible with $\ui$ and $\uj$, then every pair $\{ p,q \}$ of 
$\rho$ is contained either in $P_o$ or in a $P_{u,v}$,
with the colour of $\{ p,q \}$ being governed by the following rules:
\begin{equation}   \label{eqn:84h}
\left\{  \begin{array}{l}
\mbox{If $\{p,q\} \subseteq P_{u,v}$ for some $1 \leq u < v \leq d$,
                 then $\{p,q \}$ has colour blue.}    \\
\mbox{If $\{p,q\} \subseteq P_o$ and $\ui (p) \neq \ui (q)$,
                 then $\{p,q\}$ has colour red.}       \\
\mbox{If $\{p,q\} \subseteq P_o$ and $\ui (p) = \ui (q)$,
                 then $\{p,q\}$ may be either blue or red.}       \\
\end{array}   \right.  
\end{equation}
Thus every $\rho \in \cP_2^{\mathrm{(b-r)}} (k)$
compatible with $\ui$ and $\uj$ gets to be identified to a family
\begin{equation}    \label{eqn:84i}
\{ \rho^{(o)} \} \cup \{ \rho^{(u,v)} \mid 
1 \leq u<v \leq d \mbox{ with } P_{u,v} \neq \emptyset \}
\end{equation}
where $\rho^{(o)}$ is a bicoloured pair-partition of $P_o$, every 
$\rho^{(u,v)}$ is a pair-partition of $P_{u,v}$, and each of these
partitions must satisfy some compatibility conditions with $\ui$ and 
$\uj$ (inherited from the compatibility conditions satisfied by
$\rho$).  It is important to note that the compatibility conditions
which have to be satisfied by $\rho^{(o)}$ and the ones that have 
to be satisfied by the various $\rho^{(u,v)}$'s do not interfere with 
each other.  So what we get here is a bijection
\begin{equation}   \label{eqn:84j}
\left(  \begin{array}{c}
\rho \in \cP_2^{\mathrm{(b-r)}} (k) \\
\mbox{compatible with $\ui$ and $\uj$}
\end{array}  \right) \ \longleftrightarrow
\ \cP^{(o)} \times \prod_{ \substack{ 1 \leq u<v \leq d \\
                \mathrm{with} \ P_{u,v} \neq \emptyset } } \cP^{(u,v)},
\end{equation}
where $\cP^{(o)}$ is a certain set of bicoloured pair-partitions of 
$P_o$, and every $\cP^{(u,v)}$ is a certain set
\footnote{ Trying to give a detailed description of what are $\cP^{(o)}$ 
and $\cP^{(u,v)}$ would make the notation become quite heavy.  
We believe it is less painful and nevertheless quite convincing to see
how the bijection (\ref{eqn:84j}) and the factorization (\ref{eqn:84k})
work in a relevant concrete example, such as the one presented immediately 
following to the present proof.}
of pair-partitions of $P_{u,v}$.

Now, the bijection observed in (\ref{eqn:84j}) can be used as 
a change of variable in the indexing set of the summation 
(\ref{eqn:84g}), and this leads to a factorization:
\begin{equation}    \label{eqn:84k}
\sum_{\substack{ \rho \in \cP_2^{\mathrm{(b-r)}} (k) \\
                 \mathrm{compatible \ with} \ \ui, \uj } }
\Bigl( \, \prod_{ \substack{ \{ p,q \} \in \rho, \\ 
                  \mathrm{blue,} \ p<q } } w_{\ui (q)} \cdot
\prod_{ \substack{ \{ p,q \} \in \rho, \\ \mathrm{red,} \ p<q } }
(- w_{\ui (p)} w_{\ui (q)} ) \, \Bigr) 
= C_o' \cdot \prod_{ \substack{ 1 \leq u < v \leq d \\
            \mathrm{with} \, P_{u,v} \neq \emptyset } } C_{u,v}',
\end{equation}
where $C_o'$ is expressed as a summation over $\cP^{(o)}$ and every
$C_{u,v} '$ is expressed as a summation over $\cP^{(u,v)}$.  

The last thing which needs to be observed is that the summation
$\sum_{\cP^{(o)}}$ which gives $C_o'$ is nothing but the Wick 
expansion for the expectation 
$\varphi ( \prod_{p \in P_o} a_{\ui (p), \ui (p)} )$ that had 
appeared in (\ref{eqn:84e}); consequently, we have $C_o' = C_o$.
Likewise, for $1 \leq u < v \leq d$ with $P_{u,v} \neq \emptyset$, 
the summation $\sum_{\cP^{(u,v)}}$ which gives $C_{u,v}'$ is 
the CCR-Wick expansion for the expectation 
$\varphi ( \prod_{p \in P_{u,v}} a_{\ui (p), \uj (p)} )$ 
in (\ref{eqn:84e}); consequently, we have $C_{u,v}' = C_{u,v}$.  
We conclude that the two sides of Equation (\ref{eqn:84a}) are indeed
equal to each other, since they were expessed as
$C_o \cdot \prod_{u,v} C_{u,v}$ and
$C_o' \cdot \prod_{u,v} C_{u,v}'$, respectively.
\end{proof}

\vspace{6pt}

\begin{example}    \label{example:85}
In order to have a better grasp of what the proof of 
Proposition \ref{prop:84} is doing, it is useful to run the 
steps of the proof in the setting of a concrete example.  
Assume for instance that $k=10$ and we are looking at
\begin{equation}    \label{eqn:85a}
\varphi \bigl(
  {\color{red} a_{1,1}} \, {\color{blue} a_{1,2}} \, {\color{red} a_{1,1}}
    \, {\color{red} a_{3,3}} \, {\color{blue} a_{2,1}} 
    \, {\color{blue} a_{3,1} \, a_{1,2} \, a_{1,3}} \, {\color{red} a_{2,2}} 
    \, {\color{blue} a_{2,1}} \bigr).
\end{equation}
The tuples $\ui, \uj \in \{ 1, \ldots , d \}^{10}$ used in
this example are thus
\begin{equation}   \label{eqn:85b}
\begin{array}{ll}
             &  \ui = (1,1,1,3,2,3,1,1,2,2)  \\
\mbox{ and}  &  \uj = (1,2,1,3,1,1,2,3,2,1). 
\end{array}
\end{equation}
The set of positions in $\{ 1, \ldots , 10 \}$ where we see elements
$a_{i,i}$ is $P_o = \{ {\color{red} 1,3,4,9} \}$. We observe that we also have
$P_{1,2} = \{ {\color{blue} 2,5,7,10} \}$ 
(positions where we see an $a_{1,2}$ or an $a_{2,1}$), and
$P_{1,3} = \{ {\color{blue} 6,8} \}$ 
(positions where we see an $a_{1,3}$ or an $a_{3,1}$). 
The other sets $P_{u,v}$ defined in (\ref{eqn:84c}) of the proof of 
Proposition \ref{prop:84} (e.g.~$P_{2,3}$, or $P_{1,v}$ with 
$4 \leq v \leq d$) are empty and will simply not appear in the 
further discussion of the example.

Now let us start to follow the steps of the proof of 
Proposition \ref{prop:84}.
Commuting independence allows us to re-arrange the product of 
$a_{i,j}$'s from (\ref{eqn:85a}) as
\[
( a_{1,1} \, a_{1,1} \, a_{3,3} \, a_{2,2} ) \times
( a_{1,2} \, a_{2,1} \, a_{1,2} \, a_{2,1} ) \times
( a_{3,1} \, a_{1,3} )
\]
and to factor its expectation as $C_o \cdot C_{1,2} \cdot C_{1,3}$
where
\[
C_o = \varphi ( a_{1,1} \, a_{1,1} \, a_{3,3} \, a_{2,2} ),
\ C_{1,2} = \varphi ( a_{1,2} \, a_{2,1} \, a_{1,2} \, a_{2,1} ),
\ C_{1,3} = \varphi ( a_{3,1} \, a_{1,3} ).
\]
We note that $C_o$ can be evaluated by using the Wick formula 
reviewed in Definition \ref{def:82}.2:
\begin{align*}
C_o 
& = c_{1,1} c_{3,2} + c_{1,3} c_{1,2} + c_{1,2} c_{1,3} \\
& = ( w_1 - w_1^2 ) ( - w_2 w_3) + ( - w_1 w_3) ( - w_1 w_2 )
    + (- w_1 w_2) ( - w_1 w_3) \\
& = (3 w_1^2 - w_1) w_2 w_3,
\end{align*}
where the entries $c_{i,j}$ of the covariance matrix $C$ are as 
stated in Proposition \ref{prop:84}.  For $C_{1,2}$ and $C_{1,3}$
we use the CCR-Wick formula discussed in Proposition \ref{prop:83}, 
which gives
\[
C_{1,2} = w_2^2 + w_1 w_2 \ \mbox{ and }
\ C_{1,3} = w_1.
\]
The two terms in the formula for $C_{1,2}$ correspond to 
the pairings $\bigl\{ \{2,5\}, \, \{7,10\} \bigr\}$ and
$\bigl\{ \{2,10\}, \, \{5,7\} \bigr\}$ in 
$\cP_2 ( P_{1,2} )$, while the formula for $C_{1,3}$ uses 
the unique pairing in $\cP_2 ( P_{1,3} )$.

\vspace{6pt}

Let us now check what pairings $\rho \in \cP_2^{\mathrm{(b-r)}} (10)$
are compatible with the tuples $\ui$ and $\uj$ indicated in 
(\ref{eqn:85b}).  It is clear that the $a_{3,1}$ on position $6$ must
belong to a blue pair of $\rho$, and this pair can only be $\{6,8\}$. 
Likewise, the occurrences of $a_{1,2}$ on positions $2$ and $7$ have 
to be paired (in some order) to the occurrences of $a_{2,1}$ on positions 
$5$ and $10$, which will make for two more blue pairs of $\rho$.  Finally, 
the $a_{i,i}$'s on positions $1,3,4,9$ remain to be paired among 
themselves; this will make for some red pairs of $\rho$, with the exception
of the fact that if we pair together the two occurrences of $a_{1,1}$, 
that pair could be either red or blue.  This discussion leads precisely 
to the bijection indicated in (\ref{eqn:84j}) of Proposition \ref{prop:84}.
On the right-hand side of (\ref{eqn:84j}) we have, for this example, a
Cartesian product $\cP^{(o)} \times \cP^{(1,2)} \times \cP^{(1,3)}$ where:
\[
\left\{   \begin{array}{ll}
\bullet  &
\cP^{(o)} \mbox{ consists of the pair-partitions }
\bigl\{  {\color{blue} \{1,3\}}, {\color{red} \{4,9\}}  \bigr\},
\bigl\{  {\color{red} \{1,3\}, \{4,9\}}  \bigr\},                   \\
         &
\bigl\{  {\color{red} \{1,4\}, \{3,9\}}  \bigr\},          
\bigl\{  {\color{red} \{1,9\}, \{3,4\}}  \bigr\}
\mbox{ of } P_o.                                             \\
         &                                                   \\
\bullet  &
\cP^{(1,2)} \mbox{ consists of the pair-partitions }
\bigl\{  {\color{blue} \{2,5\}, \{7,10\}}   \bigr\},
\bigl\{  {\color{blue} \{2,10\}, \{5,7\}}   \bigr\}
\mbox{ of } P_{1,2}.                                          \\
         &                                                    \\
\bullet  &
\cP^{(1,3)} \mbox{ consists of the pair-partition }
\bigl\{ \, {\color{blue} \{6,8\}} \, \bigr\},
\mbox{ of } P_{1,3}.    
\end{array}  \right.
\]
We leave it to the reader to use the concrete descriptions of 
$\cP^{(o)}, \cP^{(1,2)}$ and $\cP^{(1,3)}$ in order to write down
explicitly what we get in the factorization (\ref{eqn:84k}). 
This factorization will produce some constants 
$C_o', C_{1,2}', C_{1,3}'$, which are found (another exercise for 
the reader) to coincide precisely with the $C_o, C_{1,2}$ and 
$C_{1,3}$ that were computed above.
\end{example}

\vspace{6pt}

\begin{ad-hoc-item}    \label{proof:86}
{\bf Proof of Theorem \ref{thm:29}.}
Since we saw in Remark \ref{rem:410} that the law $\muw$ is uniquely
determined by its moments, it will suffice to pick a $k \in \bN$ and 
verify that the moment of order $k$ of the law of $M \in M_d ( \cA )$ 
is equal to the moment of order $k$ of $\muw$.  The said moment of the 
law of $M$ is
\[
\varphi_{\uw} (M^k)
= \sum_{i=1}^d w_i 
    \, \varphi \bigl( \mbox{$(i,i)$-entry of $M^k$} \bigr)
= \sum_{i=1}^d w_i \, \varphi ( \sum_{i_2, \ldots , i_k=1}^d
a_{i,i_2} a_{i_2,i_3} \cdots a_{i_{k-1},i_k} a_{i_k,i} )  
\]
\begin{equation}    \label{eqn:86a}
= \sum_{\ui : \{ 1, \ldots ,k \} \to \{ 1, \ldots , d\} } 
w_{\ui (1)} \, \varphi \bigl( a_{\ui (1), \ui (2)}, \ldots , 
            a_{\ui (k-1), \ui (k)} a_{\ui (k), \ui (1)} \bigr).
\end{equation}

\vspace{3pt}

Now pick a tuple $\ui : \{ 1, \ldots , k \} \to \{ 1, \ldots , d \}$
and observe that if we put $\uj := \ui \circ \cyc_k$ (that is, we put 
$\uj (1) = \ui (2), \ldots , \uj (k-1) = \ui (k), \uj (k) = \ui (1)$), 
then the Wick formula (\ref{eqn:84a}) used for these $\ui$ 
and $\uj$ gives:
\begin{equation}   \label{eqn:86b}
\varphi \bigl( a_{\ui (1), \ui (2)}, \ldots , 
a_{\ui (k-1), \ui (k)} a_{\ui (k), \ui (1)} \bigr)
\end{equation}
\[
=  \sum_{\rho \in \cP_2^{\mathrm{(b-r)}} (k)}
\Bigl( \, \prod_{ \substack{ \{ p,q \} \in \rho, \\ 
                  \mathrm{blue,} \ p<q } }
\delta_{\ui (p), (\ui \circ c_k) (q)} 
\delta_{\ui (q), (\ui \circ c_k) (p)} w_{\ui (q)} \cdot
\prod_{ \substack{ \{ p,q \} \in \rho, \\ \mathrm{red,} \ p<q } }
\delta_{\ui (p), (\ui \circ c_k) (p)} 
\delta_{\ui (q), (\ui \circ c_k) (q)} 
\cdot (- w_{\ui (p)} w_{\ui (q)} ) \, \Bigr) .
\]
But for any $\rho \in \cP_2^{\mathrm{(b-r)}} (k)$, the 
definition of the permutation $\sigma_{\rho}^{\mathrm{blue}}$
can be used to evaluate the product of Kronecker $\delta$'s 
on the right-hand side of (\ref{eqn:86b}), as follows:
\[
\prod_{ \substack{ \{ p,q \} \in \rho, \\ 
                  \mathrm{blue,} \ p<q } }
\delta_{\ui (p), (\ui \circ c_k) (q)} 
\delta_{\ui (q), (\ui \circ c_k) (p)} \cdot
\prod_{ \substack{ \{ p,q \} \in \rho, \\ \mathrm{red,} \ p<q } }
\delta_{\ui (p), (\ui \circ c_k) (p)} 
\delta_{\ui (q), (\ui \circ c_k) (q)}
\]
\[
=  \prod_{ \substack{ \{ p,q \} \in \rho, \\ 
                  \mathrm{blue,} \ p<q } }
\delta_{(\ui \circ \sigma_{\rho}^{\mathrm{blue}}) (q), (\ui \circ c_k) (q)} 
\delta_{(\ui \circ \sigma_{\rho}^{\mathrm{blue}}) (p), (\ui \circ c_k) (p)} \cdot
\prod_{ \substack{ \{ p,q \} \in \rho, \\ \mathrm{red,} \ p<q } }
\delta_{(\ui \circ \sigma_{\rho}^{\mathrm{blue}}) (p), (\ui \circ c_k) (p)} 
\delta_{(\ui \circ \sigma_{\rho}^{\mathrm{blue}}) (q), (\ui \circ c_k) (q)}
\]
\[
= \prod_{h=1}^k 
\delta_{(\ui \circ \sigma_{\rho}^{\mathrm{blue}}) (h), (\ui \circ c_k) (h)} 
= \left\{ \begin{array}{ll}
1,  &  \mbox{ if 
        $\ui \circ \sigma_{\rho}^{\mathrm{blue}} = \ui \circ \cyc_k$},     \\
0,  &   \mbox{ otherwise.}
\end{array}   \right.
\]

The conclusion of the preceding paragraph is that, for every 
$\ui : \{ 1, \ldots , k \} \to \{ 1, \ldots , d \}$, Equation
(\ref{eqn:86b}) evaluates the expectation
$\varphi \bigl( a_{\ui (1), \ui (2)}, \ldots , 
a_{\ui (k-1), \ui (k)} a_{\ui (k), \ui (1)} \bigr)$ to
\begin{equation}   \label{eqn:86c}
\sum_{ \substack{ \rho \in \cP_2^{\mathrm{(b-r)}} (k)  \\
                     \mathrm{with} \ \ui \circ c_k = 
                             \ui \circ \sigma_{\rho}^{\mathrm{blue}} } }
\Bigl( \, \prod_{ \substack{ \{ p,q \} \in \rho, \\ 
                  \mathrm{blue,} \ p<q } }  w_{\ui (q)} \cdot
\prod_{ \substack{ \{ p,q \} \in \rho, \\ \mathrm{red,} \ p<q } }
(- w_{\ui (p)} w_{\ui (q)} ) \, \Bigr) .
\end{equation}
When substituting this into (\ref{eqn:86a}), we find that 
$\varphi_{\uw} (M^k)$ is precisely given by the formula we had 
found in Proposition \ref{prop:76} for the moment of order $k$ of $\muw$.
\hfill  $\square$
\end{ad-hoc-item}

$\ $

\end{document}